\documentclass{amsart}
\usepackage[letterpaper, margin=1.4in]{geometry}
\usepackage{amssymb, latexsym, MnSymbol}
\usepackage{color, xcolor}
\usepackage{amscd, graphicx, psfrag, tikz}
\usepackage[mathscr]{euscript}
\usepackage[all]{xy}

\usepackage{enumitem}
\usepackage{stmaryrd}
\usepackage{url,hyperref}
\hypersetup{
    colorlinks=true,
    linkcolor=black,
    citecolor=black,
    urlcolor=black
}
\usepackage{verbatim}

\newcommand{\Si}{ {\Sigma} }

\newcommand{\ep}{ {\epsilon} }

\newcommand{\bC}{ {\mathbb{C}} }

\newcommand{\bL}{\mathbb{L}}

\newcommand{\bP}{\mathbb{P}}
\newcommand{\bQ}{\mathbb{Q}}
\newcommand{\bR}{\mathbb{R}}

\newcommand{\bZ}{\mathbb{Z}}

\newcommand{\cA}{\mathcal{A}}

\newcommand{\cF}{\mathcal{F}}

\newcommand{\cM}{\mathcal{M}}
\newcommand{\cO}{\mathcal{O}}

\newcommand{\Aut}{\mathrm{Aut}}

\newcommand{\Hom}{\mathrm{Hom}}
\newcommand{\Ext}{\mathrm{Ext}}

\newcommand{\ev}{\mathrm{ev}}
\newcommand{\val}{ {\mathrm{val}} }
\newcommand{\vir}{ {\mathrm{vir}} }

\newcommand{\pt}{\mathrm{pt}}

\newcommand{\et}{ {\mathrm{\acute{e}t}}}

\newcommand{\bh}{\mathbf{h}}
\newcommand{\bw}{\mathbf{w}}
\newcommand{\bfD}{\mathbf{D}}
\newcommand{\bfR}{\mathbf{R}}

\newcommand{\fg}{\mathfrak{g}}

\newcommand{\su}{\mathsf{u}}

\newcommand{\sw}{\mathsf{w}}

\newcommand{\hv}{\hat{v}}

\newcommand{\hl}{\hat{l}}

\newcommand{\hSi}{\hat{\Sigma}}
\newcommand{\hsi}{\hat{\sigma}}
\newcommand{\hbeta}{\hat{\beta}}
\newcommand{\hpsi}{\hat{\psi}}

\newcommand{\htau}{\hat{\tau}}
\newcommand{\hvphi}{\hat{\varphi}}

\newcommand{\hA}{\hat{A}}
\newcommand{\hC}{\hat{C}}

\newcommand{\hbL}{\hat{\bL}}

\newcommand{\tSi}{\widetilde{\Sigma}}
\newcommand{\trho}{\widetilde{\rho}}

\newcommand{\tbL}{\widetilde{\bL}}

\newcommand{\tC}{\widetilde{C}}
\newcommand{\tD}{\widetilde{D}}

\newcommand{\tM}{\widetilde{M}}
\newcommand{\tN}{\widetilde{N}}

\newcommand{\tT}{\widetilde{T}}

\newcommand{\tX}{\widetilde{X}}

\newcommand{\tZ}{\widetilde{Z}}

\newcommand{\tb}{\widetilde{b}}

\newcommand{\tl}{\widetilde{l}}

\newcommand{\tu}{\widetilde{u}}

\newcommand{\tsi}{\widetilde{\sigma}}
\newcommand{\tbeta}{\widetilde{\beta}}
\newcommand{\tpsi}{\widetilde{\psi}}
\newcommand{\tvphi}{\widetilde{\varphi}}
\newcommand{\ttau}{\widetilde{\tau}}

\newcommand{\tbh}{\widetilde{\bh}}

\newcommand{\tbw}{\widetilde{\bw}}

\newcommand{\vGa}{\vec{\Gamma}}
\newcommand{\vd}{\vec{d}}
\newcommand{\vf}{\vec{f}}

\newcommand{\vs}{\vec{s}}

\newcommand{\Mbar}{\overline{\cM}}

\newcommand{\lra}{\longrightarrow}

\newcommand{\inner}[1]{\langle  #1 \rangle}

\newtheorem{dummy}{dummy}[section]
\newtheorem{lemma}[dummy]{Lemma}
\newtheorem{theorem}[dummy]{Theorem}

\newtheorem{proposition}[dummy]{Proposition}
\newtheorem{remark}[dummy]{Remark}
\newtheorem{definition}[dummy]{Definition}
\newtheorem{example}[dummy]{Example}

\newtheorem{assumption}[dummy]{Assumption}

\begin{document}
\title{Open/closed correspondence via relative/local correspondence}

\author{Chiu-Chu Melissa Liu}
\address{Chiu-Chu Melissa Liu, Department of Mathematics, Columbia University, 2990 Broadway, New York, NY 10027}
\email{ccliu@math.columbia.edu}

\author{Song Yu}
\address{Song Yu, Department of Mathematics, Columbia University, 2990 Broadway, New York, NY 10027}
\email{syu@math.columbia.edu}

\dedicatory{Dedicated to the memory of Professor Bumsig Kim}

\begin{abstract}
We establish a correspondence between the disk invariants of a smooth toric Calabi-Yau 3-fold $X$ with boundary condition specified by a framed Aganagic-Vafa outer brane $(L, f)$ and the genus-zero closed Gromov-Witten invariants of a smooth toric Calabi-Yau 4-fold $\tX$, proving the open/closed correspondence proposed by Mayr and developed by Lerche-Mayr. Our correspondence is the composition of two intermediate steps:

\begin{itemize}
    \item First, a correspondence between the disk invariants of $(X,L,f)$ and the genus-zero maximally-tangent relative Gromov-Witten invariants of a relative Calabi-Yau 3-fold $(Y,D)$, where $Y$ is a toric partial compactification of $X$ by adding a smooth toric divisor $D$. This correspondence can be obtained as a consequence of the topological vertex (Li-Liu-Liu-Zhou) and Fang-Liu where the all-genus open Gromov-Witten invariants of $(X,L,f)$ are identified with the formal relative Gromov-Witten invariants of the formal completion of $(Y,D)$ along the toric 1-skeleton. Here, we present a proof without resorting to formal geometry.

    \item Second, a correspondence in genus zero between the maximally-tangent relative Gromov-Witten invariants of $(Y,D)$ and the closed Gromov-Witten invariants of the toric Calabi-Yau 4-fold $\tX = \cO_Y(-D)$. This can be viewed as an instantiation of the log-local principle of van Garrel-Graber-Ruddat in the non-compact setting.
\end{itemize}
\end{abstract}
\maketitle

\section{Introduction}\label{sect:Intro}

The \emph{open/closed correspondence}, proposed by Mayr \cite{Mayr01} as a class of open/closed string dualities and developed by Lerche-Mayr \cite{LM01}, is a conjectural relation between the topological amplitudes at genus zero of an open string geometry on a Calabi-Yau 3-fold $X$ relative to a Lagrangian $L$ and a closed string geometry on a corresponding Calabi-Yau 4-fold $\tX$. From the viewpoint of Gromov-Witten theory, this predicts a correspondence between the disk invariants of $(X,L)$, which are virtual counts of stable maps from genus-zero Riemann surfaces with one boundary component to $(X,L)$, and the genus-zero closed Gromov-Witten invariants of $\tX$, which are virtual counts of stable maps from genus-zero closed Riemann surfaces to $\tX$. This numerical correspondence is conjectured to situate in a correspondence between the generating functions for the two types of invariants under an identification of the K\"ahler parameters of $\tX$ with the open/closed moduli parameters of $X$. Moreover, under mirror symmetry, especially that for open strings introduced by Aganagic-Vafa \cite{AV00}, the correspondence predicts an identification of the complex moduli of the 4-fold family mirror to $\tX$ with the open-closed phase space of the 3-fold family mirror to $X$, under which a correspondence between period integrals on the two mirror families should hold. The two types of periods are conjectured to satisfy the same system of differential equations, which can be constructed from either the open geometry of $(X,L)$ or the closed geometry of $\tX$.

In this work, we set stage for a mathematical treatment of the above proposal by focusing on the numerical correspondence. Given a smooth toric Calabi-Yau 3-fold $X$ and a Lagrangian submanifold $L$ of Aganagic-Vafa type, we give an explicit construction of the corresponding toric Calabi-Yau 4-fold $\tX$. We then formulate and prove a precise correspondence between the disk invariants of $(X,L)$ and the genus-zero closed Gromov-Witten invariants of $\tX$.

\subsection{Main result: open/closed correspondence}
Let $X$ be a smooth toric Calabi-Yau 3-fold, $L \subset X$ be an Aganagic-Vafa brane, and $f \in \bZ$ be a framing on the brane $L$. We assume that $L$ is \emph{outer}, intersecting a unique non-compact torus invariant line $l$ in $X$ and bounding a disk $B$ in $l$. We further assume that $L$ is an outer brane in a semi-projective partial compactification of $X$, and that choice of the framing $f$ is generic. Given any effective class
$$
    \beta  = \beta' + d[B] \in H_2(X,L;\bZ)
$$
with $\beta' \in H_2(X;\bZ)$,  the virtual/expected dimension of $\cM_\beta =\Mbar(X,L\mid \beta',d)$,  the moduli of degree-$\beta$ stable maps from genus-zero bordered Riemann surfaces with a single boundary component to $(X,L)$, is zero. The \emph{disk invariant}
$$
N^{X,L,f}_{\beta', d}  
$$
is the virtual number of  points in $\cM_\beta$; it is defined  by torus localization on $\cM_\beta$ and depends on a generic circle action on the pair $(X,L)$ specified by the framing $f$. (The precise definition of  $N^{X,L,f}_{\beta'd}$ will be given in Section \ref{sec:disk-invariants}.)

The disk invariants are a special class of \emph{open Gromov-Witten invariants} of $(X,L,f)$ \cite{FL13,KL01,LLLZ09} which encode stable maps from domains of possibly higher genus and with multiple boundary components at which a general winding profile is realized.  Our main result is a correspondence between the disk invariants of $(X,L,f)$ and the genus-zero closed Gromov-Witten invariants of a toric Calabi-Yau 4-fold, establishing the open/closed correspondence at the numerical level:

\begin{theorem}[Open/closed correspondence]\label{thm:OpenClosedIntro}
Given $X, L, f$ as above, there is a smooth toric Calabi-Yau 4-fold $\tX$ such that
    \begin{itemize}
        \item There is an isomorphism $H_2(X,L;\bZ) \cong H_2(\tX;\bZ)$.

        \item Given any effective class $\beta' \in H_2(X;\bZ)$ of $X$ and $d \in \bZ_{>0}$, if $\tbeta \in H_2(\tX;\bZ)$ is the effective class of $\tX$ corresponding to $\beta := \beta' + d[B] \in H_2(X,L;\bZ)$ under the above isomorphism, we have
        \begin{equation}\label{eqn:OpenClosedIntro}
            N^{X, L, f}_{\beta', d} = dN^{\tX}_{\tbeta,0} = N^{\tX}_{\tbeta,1},
        \end{equation}
        where $N^{\tX}_{\tbeta,0}$ (resp. $N^{\tX}_{\tbeta,1}$) is a suitably defined genus-zero, $0$-pointed (resp. $1$-pointed), degree-$\tbeta$ closed Gromov-Witten invariant of $\tX$.
    \end{itemize}
\end{theorem}

The construction of the toric Calabi-Yau 4-fold $\tX$ proceeds as follows: We first partially compactify $X$ into a smooth toric 3-fold $Y$ by adding an irreducible toric divisor $D$ whose position depends on the position of the Aganagic-Vafa brane $L$. In particular, our newly-added divisor $D$ contains a new torus-fixed point $p$ that compactifies the torus-invariant line $l$ that $L$ intersects into a $\bP^1$, which we denote by $\bar{l} = l \cup \{p\}$. There is a natural isomorphism $H_2(Y;\bZ) \cong H_2(X,L;\bZ)$, which in particular identifies $[\bar{l}]$ with $[B]$. Then, we take $\tX$ to be the total space of the anti-canonical line bundle $\cO_Y(-D)$. 
The inclusion of the zero section $Y \hookrightarrow \tX$ induces an isomorphism $H_2(Y;\bZ) \cong H_2(\tX;\bZ)$ that identifies effective curve classes. The virtual dimension of 
$\Mbar_{0,n}(\tX,\tbeta)$, the moduli  of genus zero, $n$-pointed, degree $\tbeta$ stable maps to $\tX$, is $n+1$. The $1$-point primary closed Gromov-Witten invariant $N^{\tX}_{\tbeta,1}$ is
the virtual number of maps in $\Mbar_{0,1}(\tX,\tbeta)$ which send  the unique marked point into a toric surface $S\subset \tX$ (which is a codimension 2 constraint).

We establish the correspondence \eqref{eqn:OpenClosedIntro} by relating the invariants to the genus-zero \emph{maximally-tangent relative Gromov-Witten invariants} of $(Y,D)$. Given any effective class $\hbeta \in H_2(Y;\bZ)$ such that $d:= \hbeta \cdot D >0$,  the virtual dimension of $\Mbar(Y/D,\hbeta)$, the moduli  of degree-$\hbeta$ relative stable maps from genus-zero Riemann surfaces with a single marked point to $Y$, such that the contact order of the marked point with $D$ is $d$, the maximum possible, is one. The genus-zero, degree-$\hbeta$ maximally-tangent relative Gromov-Witten invariant 
$$
    N^{Y,D}_{\hbeta}
$$
is the virtual number of maps in $\Mbar(Y/D,\hbeta)$ such that the marked point lies in a complex line $\bC\subset D$ (which is a codimension 1 constraint). 
The maximally-tangent invariants are a special class of relative Gromov-Witten invariants $(Y,D)$ \cite{Li01,Li02} which encode stable maps from domains of possibly higher genus and with multiple marked points at which a general tangency profile is realized.

We note that since our spaces $X, Y, \tX$ are non-compact, the three types of Gromov-Witten invariants involved here are defined via virtual localization \cite{B97, GP99, GV05} and appropriate weight restrictions.

\subsection{The basic example} 
The most basic example is  $X=\bC^3$. In this case  
$$
Y =\mathrm{Tot}\left( \cO_{\bP^1}(f)\oplus \cO_{\bP^1}(-f-1) \right) ,\quad
\tX =\mathrm{Tot}\left( \cO_{\bP^1}(f) \oplus \cO_{\bP^1}(-f-1) \oplus \cO_{\bP^1}(-1) \right), 
$$
and $D\cong \bC^2$ is the fiber of the projection $Y\lra \bar{l}=\bP^1$ over $p$.  

We have  $\beta'=0$ and $\hbeta= d\bar{l}$ for some positive integer $d$.  
The disk invariant
$$
N^{\bC^3,L,f}_{0,d} = \frac{1}{d^2} (-1)^{fd} \frac{\prod_{k=1}^{d-1}(fd+k)}{(d-1)!} 
$$
gives  a virtual count of degree-$d$ disk covers  of an embedded  disk in $\bC^3$ bounded by $L$, and 
specializes to the famous formula $1/d^2$ at the zero framing $f=0$. Moreover, when $f=0$, 
$$
Y = \bC \times S,\quad
\tX = \bC \times  X',
$$
where  $S=\mathrm{Tot}(\cO_{\bP^1}(-1))$ is the blowup of $\bC^2$ at the origin, and 
$X'=\mathrm{Tot}(\cO_{\bP^1}(-1)^{\oplus 2})$ is the resolved conifold. 
We observe that
$$
 \Mbar_{0,n}(\tX,d\bar{l}) = \bC \times \Mbar_{0,n}(X',d\bar{l}),\quad
\Mbar(Y/D,d\bar{l}) =  \bC\times \Mbar(S/F,d\bar{l}),
 $$
 where $F\cong \bC$ is the fiber of the projection $S\lra \bar{l}=\bP^1$ over $p$.  
The closed Gromov-Witten invariants  $N^{\tX}_{d\bar{l},0}$ and  $N^{\tX}_{d\bar{l},1}$ of $\tX$ are reduced to
well-known genus-zero degree-$d$ closed Gromov-Witten invariants of the resolved conifold:  
$$
N^{\tX}_{d\bar{l},0}  = \langle \;  \rangle^{X'}_{0,0, d\bar{l}}  =\frac{1}{d^3},\quad N^{\tX}_{d\bar{l},1}  =\langle H \rangle^{X'}_{0,1, d\bar{l}} =\frac{1}{d^2}. 
$$
Here $1/d^3$ is the Aspinwall-Morrison formula \cite{AM93} of genus-zero, degree-$d$ covers of the zero section $\bP^1\subset X'$, $H$ is the hyperplane class, and 
$\langle H\rangle^{X'}_{0,1, d\bar{l}} = d \langle\; \rangle^{X'}_{0,0, d\bar{l}}$ by the divisor equation.  
The  relative Gromov-Witten invariant
$$
N^{Y,D}_{d\bar{l}} =\frac{(-1)^{d+1}}{d^2}
$$
of the pair $(Y,D)$  coincides with  the relative/log Gromov-Witten invariant of the pair $(S,F)$ computed in the proof of Proposition 2.4 in \cite{vGGR19}.

\subsection{Main ingredient I: open/relative correspondence}
The first step in proving Theorem \ref{thm:OpenClosedIntro} is a correspondence between the disk invariants of $(X,L,f)$ and the genus-zero maximally-tangent relative Gromov-Witten invariants of $(Y,D)$: 

\begin{theorem}[Open/relative correspondence \cite{FL13, LLLZ09}]\label{thm:OpenRelativeIntro}
Let $\beta' \in H_2(X;\bZ)$ be an effective class of $X$, $d \in \bZ_{>0}$, and $\hbeta \in H_2(Y;\bZ)$ be the effective class of $Y$ corresponding to $\beta := \beta' + d[B]$ under the isomorphism $H_2(X,L;\bZ) \cong H_2(Y;\bZ)$. Then
$$
    N^{X,L,f}_{\beta', d} = (-1)^{d + 1}N^{Y,D}_{\hbeta}.
$$
\end{theorem}

Theorem \ref{thm:OpenRelativeIntro} is a special case of a general correspondence between open Gromov-Witten invariants of $(X,L,f)$ and relative Gromov-Witten invariants of $(Y,D)$, which involves invariants of higher genus and general winding/tangency profiles. The general open/relative correspondence is already established by Fang-Liu \cite{FL13}, whose proof builds upon a relation between the open Gromov-Witten invariants of $(X,L,f)$ and the \emph{formal} relative Gromov-Witten invariants of the $(\hat{Y}, \hat{D})$, the formal completion of $(Y,D)$ along the toric 1-skeleton of $Y$. Formal relative Gromov-Witten invariants for a general formal toric Calabi-Yau 3-fold relative to a collection of boundary divisors are introduced by Li-Liu-Liu-Zhou \cite{LLLZ09} as a fundamental building block for a mathematical theory of the topological vertex. In our case, since the relative condition is specified by an irreducible divisor $D$, we can directly define and compute the relative invariants without resorting to formal geometry. This also simplifies our comparison between the relative invariants and the open invariants while proving Theorem \ref{thm:OpenRelativeIntro}.

\subsection{Main ingredient II: relative/local correspondence}
The second step in proving Theorem \ref{thm:OpenClosedIntro} is a correspondence in genus zero between the maximally-tangent relative Gromov-Witten invariants of $(Y,D)$ and the closed Gromov-Witten invariants of $\tX$, which we can view as a \emph{local} Calabi-Yau 4-fold over the base $Y$:

\begin{theorem}[Relative/local correspondence]\label{thm:RelativeLocalIntro}
Let $\hbeta \in H_2(Y; \bZ)$ be an effective class of $Y$ and $\tbeta \in H_2(X;\bZ)$ be the corresponding effective class of $\tX$, such that $d:= \hbeta \cdot D >0$. Then
$$
    N^{Y,D}_{\hbeta} = (-1)^{d + 1}d N^{\tX}_{\tbeta,0} = (-1)^{d + 1} N^{\tX}_{\tbeta,1}.
$$
\end{theorem}

Theorem \ref{thm:RelativeLocalIntro} can be viewed as an instantiation of the \emph{log-local principle} of van Garrel-Graber-Ruddat \cite{vGGR19}, which at the numerical level conjectures a correspondence between the maximally-tangent \emph{log} Gromov-Witten invariants of a projective variety $Y$ relative to a normal crossing divisor $D$ whose irreducible components $D_1, \dots, D_k$ are smooth and nef, and the closed Gromov-Witten invariants of the total space of the vector bundle $\cO_Y(-D_1) \oplus \cdots \oplus \cO_Y(-D_k)$. Since its proposal, this principle has been verified in various cases and studied from different perspectives; see e.g. \cite{BBvG19, BBvG20, BBvG20b, BFGW21, BS22, CvGKT21, NR19, TY20}. It is also known not to hold in general when $D$ is reducible \cite{NR19}. Our Theorem \ref{thm:RelativeLocalIntro} provides a general class of examples for the log-local principle in the extended setting where the base $Y$ is non-compact, in view of the identification of the log and relative Gromov-Witten invariants \cite{AMW14}.

Moreover, we note that recent works of Bousseau-Brini-van Garrel \cite{BBvG20, BBvG20b} present examples for the open/closed correspondence that originate from applying the log-local principle to log Calabi-Yau surfaces with maximal boundary, or \emph{Looijenga pairs}. To certain Looijenga pairs $(Y,D)$, where $Y$ could be an orbifold, they show that one can associate a local toric Calabi-Yau 3-fold $X$ with an Aganagic-Vafa brane $L$, as well as a local Calabi-Yau 4-fold $\tX$, such that there is a correspondence between the disk invariants of $(X,L)$ and the genus-zero closed Gromov-Witten invariants of $\tX$. Our Theorem \ref{thm:OpenClosedIntro} generalizes the above examples, at least in the smooth case, in that it applies to general toric Calabi-Yau 3-folds that are not necessarily local surfaces. We note that Conjecture 1.1 of \cite{BBvG20} proposes a variant of the log-local principle for more general log Calabi-Yau pairs, and our Theorem \ref{thm:RelativeLocalIntro} verifies this conjecture for the toric pairs $(Y,D)$ arising from our constructions.

\subsection{Our techniques}
Let $\beta' \in H_2(X;\bZ)$ be an effective class of $X$, $d \in \bZ_{>0}$, $\hbeta \in H_2(Y;\bZ)$ be the effective class of $Y$ corresponding to $\beta:= \beta' + d[B]$ under the isomorphism $H_2(X,L;\bZ) \cong H_2(Y;\bZ)$, and $\tbeta \in H_2(\tX:\bZ)$ be the corresponding effective class of $\tX$. Since our spaces $X, Y, \tX$ are non-compact, we define and compare the three types of Gromov-Witten invariants involved in our correspondences via virtual localization. Specifically:
\begin{itemize}
    \item For the disk invariant $N^{X,L,f}_{\beta',d}$, we adopt the definition of Fang-Liu \cite{FL13}, which is based on the moduli space $\Mbar(X, L \mid \beta', d)$ of open stable maps of Katz-Liu \cite{KL01}. The Calabi-Yau condition on $X$ specifies a 2-dimensional subtorus $T'$ of the algebraic 3-torus $T$ in $X$ such that the action of the maximal compact subtorus $T'_{\bR} \subset T'$ on $X$ preserves $L$. $N^{X,L,f}_{\beta',d}$ is defined and computed by $T'_{\bR}$-equivariant localization on $\Mbar(X, L \mid \beta', d)$.

    \item For the maximally-tangent relative Gromov-Witten invariant $N^{Y,D}_{\hbeta}$, we use the moduli space $\Mbar(Y/D, \hbeta)$ of relative stable maps of Li \cite{Li01, Li02}, which includes maps to $(Y,D)$ as well as \emph{expanded targets} $(Y[m], D_{(m)})$ that are deformations of $(Y,D)$. $N^{Y,D}_{\hbeta}$ is defined and computed by $T'$-equivariant localization on $\Mbar(Y/D, \hbeta)$.

    \item The closed Gromov-Witten invariants $N^{\tX}_{\tbeta,0}$ and $N^{\tX}_{\tbeta,1}$ are defined and computed by $\tT'$-equivariant localization on the usual moduli spaces $\Mbar_{0,0}(\tX,\tbeta)$ and $\Mbar_{0,1}(\tX,\tbeta)$ of stable maps, where $\tT'$ is a 3-dimensional subtorus of the algebraic 4-torus $\tT$ in $\tX$ specified by the Calabi-Yau condition. We show that the two invariants are related by
    $$
        N^{\tX}_{\tbeta,1} = d N^{\tX}_{\tbeta,0}.
    $$
    Thus for the rest of this section, we focus on the relation of $N^{\tX}_{\tbeta,0}$ to the open and relative invariants.
\end{itemize}

To establish the correspondences, we directly relate the components of the torus-fixed loci of the above moduli spaces and compare the contributions of the corresponding components to the Gromov-Witten invariants in the localization computation. The relation among components can be described by the following relation among stable maps representing torus-fixed points in the moduli space: Given an open stable map $u: (C, \partial C) \to (X,L)$ that represents a point in $\Mbar(X, L \mid \beta', d)^{T'_{\bR}}$, we can write $C = C' \cup C''$ where $C'$ is a closed Riemann surface of arithmetic genus $0$ and $C''$ is a disk attached to $C'$ at a node $q$ of $C$. Then $u|_{C''}$ maps $C''$ as a degree-$d$ cover onto the disk $B$ bounded by the Aganagic-Vafa brane $L$. Now, $u$ induces a map $\tu: \tC = C' \cup \tC'' \to Y \subset \tX$, where $\tC'' \cong \bP^1$ is attached to $C'$ at $q$ and mapped to $\bar{l} \cong \bP^1$ as a degree-$d$ cover. See Figure \ref{fig:RelateMaps} for an illustration of this construction. If $x \in \tC''$ is the unique point in $\tu^{-1}(p)$, then $\tu:(\tC, x) \to (Y,D)$ is a relative stable map representing a point in $\Mbar(Y/D, \hbeta)^{T'}$. Moreover, $\tu: \tC \to \tX$ is a stable map representing a point in $\Mbar_{0,0}(\tX,\tbeta)^{\tT'}$. We then identify the component of $\Mbar(X, L \mid \beta', d)^{T'_{\bR}}$ containing $[u]$ with the component of $\Mbar(Y/D, \hbeta)^{T'}$ (or $\Mbar_{0,0}(\tX,\tbeta)^{\tT'}$) containing $[\tu]$ and explicitly compare their localization contributions.

\begin{figure}[h]
      \begin{tikzpicture}[scale=.7]
          \coordinate (1) at (-3,0);
          \node at (1) {$\bullet$};
          \draw (1) .. controls (-1, 0.9) and (1, 1) .. (2,0.9);
          \draw (1) .. controls (-1, -0.9) and (1, -1) .. (2,-0.9);
          \draw[thick] (0, 0.86) to[bend left] (0,-0.86);
          \draw[dashed, thick] (0, 0.86) to[bend right] (0,-0.86);
          \draw[dashed] (-0.7, 3) -- (0.7,2) -- (0.7, -3) -- (-0.7, -2) -- (-0.7, 3);

          \node at (2, 0) {$l$};
          \node at (-0.3, 2.1) {$L$};
          \node at (-1.5, 0) {$B$};
          \node at (-4.5, -3) {$X$};

          \draw (1) .. controls (-3.5, -0.9) and (-3.5, -1.5) .. (-3.5, -2);
          \draw (1) .. controls (-2.5, -0.9) and (-2.5, -1.5) .. (-2.5, -2);
        
          \draw (1) .. controls (-3.1, 0.7) and (-3.5,1.2) .. (-4, 1.7);
          \draw (1) .. controls (-3.5, 0) and (-4, 0.2) .. (-4.8, 0.7);

          \draw (-3, 6) .. controls (-1.8, 6.6) and (-1.2, 6.7) .. (0, 6.8);
          \draw (-3, 6) .. controls (-1.8, 5.4) and (-1.2, 5.3) .. (0, 5.2);
          \draw[thick] (0, 6.8) to[bend left] (0,5.2);
          \draw[dashed, thick] (0, 6.8) to[bend right] (0,5.2);
          \draw (-3,6) arc (0:60:2);
          \draw (-3,6) arc (0:-60:2);
          \node at (-3,6) {$\bullet$};
          \node at (-2.8, 5.5) {$q$};
          \node at (-4, 6) {$C'$};
          \node at (-1, 6) {$C''$};
          \node at (-2, 7.5) {$C$};

          \draw[->] (-2, 4) -- (-2, 2);
          \node at (-2.5, 3) {$u$};


          \node at (4,3) {$\Rightarrow$};

          \coordinate (3) at (7,0);
          \coordinate (4) at (13,0);
          \node at (3) {$\bullet$};
          \node at (4) {$\bullet$};
          \draw (3) to[bend right] (4);
          \draw (3) to[bend left] (4);
          \draw[thick] (10, 0.86) to[bend left] (10,-0.86);
          \draw[dashed, thick] (10, 0.86) to[bend right] (10,-0.86);
          \draw[dashed] (11.5, 4) -- (15,3) -- (15, -4) -- (11.5, -3) -- (11.5, 4);

          \node at (12.7, -0.8) {$p$};
          \node at (9, 0) {$\bar{l}$};

          \node at (7.5, -3) {$Y \subset \tX$};
          \node at (14.5, 2.5) {$D$};

          \draw (3) .. controls (6.5, -0.9) and (6.5, -1.5) .. (6.5, -2);
          \draw (3) .. controls (7.5, -0.9) and (7.5, -1.5) .. (7.5, -2);
        
          \draw (3) .. controls (6.9, 0.7) and (6.5,1.2) .. (6, 1.7);
          \draw (3) .. controls (6.5, 0) and (6, 0.2) .. (5.2, 0.7);
          
          \draw (4) .. controls (12.5, 0.9) and (12.5, 1.5) .. (12.5, 2);
          \draw (4) .. controls (13.5, 0.9) and (13.5, 1.5) .. (13.5, 2);

          \draw (4) .. controls (13.1, -0.7) and (13.5,-1.2) .. (14, -1.7);
          \draw (4) .. controls (13.5, 0) and (14, -0.2) .. (14.8, -0.7);

          \draw (7, 6) to[bend left] (13, 6);
          \draw (7, 6) to[bend right] (13, 6);
          \draw[thick] (10, 6.9) to[bend left] (10,5.1);
          \draw[dashed, thick] (10, 6.9) to[bend right] (10,5.1);
          \draw (7,6) arc (0:60:2);
          \draw (7,6) arc (0:-60:2);
          \node at (7,6) {$\bullet$};
          \node at (7.2, 5.5) {$q$};
          \node at (6, 6) {$C'$};
          \node at (9, 6) {$\tC''$};
          \node at (12, 7.5) {$\tC$};
          \node at (13,6) {$\bullet$};
          \node at (13.5,6) {$x$};

          \draw[->] (10, 4) -- (10, 2);
          \node at (9.5, 3) {$\tu$};

        \end{tikzpicture}
\caption{Constructing the stable map $\tu$ to $Y \subset \tX$ representing a $T'$-fixed (or $\tT'$-fixed) point in the moduli from the open stable map $u$ to $(X,L)$ representing a $T'_{\bR}$-fixed point in the moduli.}
\label{fig:RelateMaps}
\end{figure}
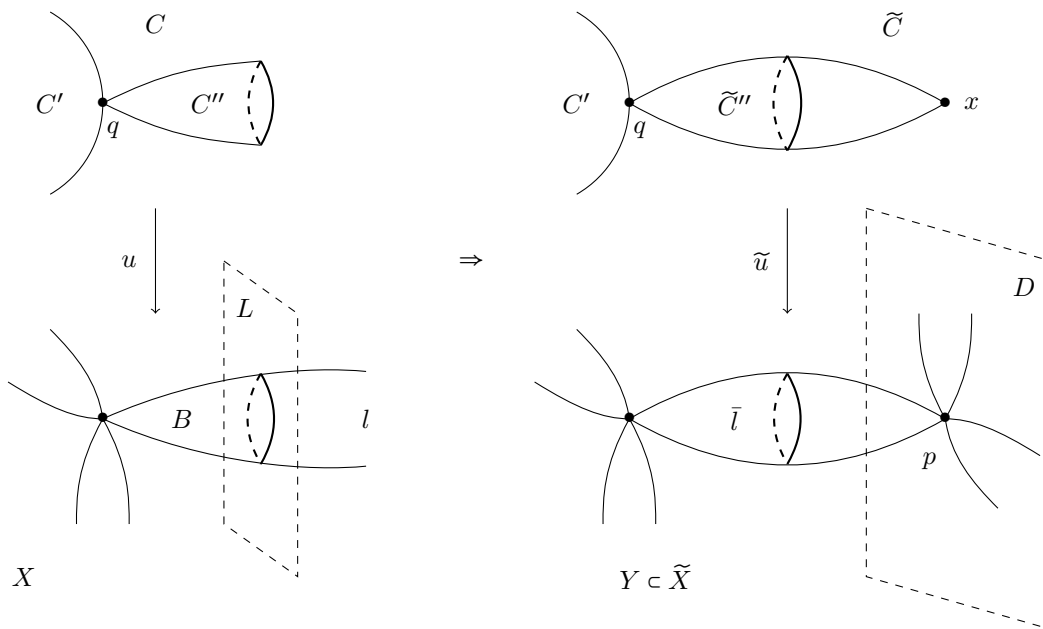

Additional care is required to complete the argument. For relative invariants of $(Y,D)$, $\Mbar(Y/D, \hbeta)$ contains points represented by stable maps to an \emph{expanded target} $(Y[m],D_{(m)})$. We show, however, that under our specialization of $T'$-weights in localization, there is no contribution to $N^{Y,D}_{\hbeta}$ from components that correspond to maps to an expanded target. For closed invariants of $\tX$, there are additional components of $\Mbar_{0,0}(\tX,\tbeta)^{\tT'}$ that do not correspond to components of $\Mbar(X, L \mid \beta', d)^{T'_{\bR}}$ in the way described above. For instance, in a stable map to $\tX$ representing a $\tT'$-fixed point, there could be two irreducible $\bP^1$'s in the domain mapping to $\bar{l}$ as covers of degrees $d_1, d_2>0$ respectively, such that $d_1 + d_2 = d$. We show that under our specialization of $\tT'$-weights in localization, none of such additional components contribute to $N^{\tX}_{\tbeta,0}$. Therefore, the only contributing components are the ones corresponding to components of $\Mbar(X, L \mid \beta', d)^{T'_{\bR}}$.

\subsection{Generalizations and implications}
The numerical open/closed correspondence (Theorem \ref{thm:OpenClosedIntro}) that we establish in this work has various generalizations and implications, which develop and possibly extend the original proposal of Mayr and Lerche-Mayr \cite{LM01,Mayr01}. We briefly indicate some of those here and defer a rigorous treatment to future work.

\subsubsection{Marked points, descendant invariants, generating functions}
It would be straightforward to generalize our correspondences to incorporate Gromov-Witten invariants with marked points as well as descendant invariants, by generalizing our localization computations and comparisons. Moreover, the correspondences between invariants at corresponding classes could be pieced together into correspondences at the level of generating functions.

\subsubsection{Mirror symmetry and B-model identifications}
In the original proposal \cite{LM01,Mayr01}, the open/closed correspondence is  formulated and studied in the context of mirror symmetry for toric Calabi-Yau manifolds constructed
by symplectic reduction. As complex algebraic varieties, they are semi-projective toric varieties, which are GIT quotients of a vector
space $\bC^R$ by the linear action of a subtorus $G\subset (\bC^*)^R$.

Assume that $X$ is semi-projective. Let $\tZ$ be a suitable semi-projective toric Calabi-Yau 4-fold that partially compactifies $\tX$. Then on the mirror family of $\tZ$, period integrals of the holomorphic $4$-form are solutions to a system of differential equations $PF_{\tZ}$ known as the \emph{Picard-Fuchs system} \cite{HV00}, which can be described by the toric data of $\tZ$ (or of $\tX$). On the open counterpart, on the mirror family of $X$, period integrals of the holomorphic $3$-form on \emph{relative} cycles with prescribed boundary conditions are solutions to an \emph{extended} Picard-Fuchs system $PF_X$ that incorporates the open moduli parameter \cite{AV00}, which can be described by the data of $(X,L,f)$. The open/closed correspondence then proposes an identification of the two Picard-Fuchs systems $PF_{\tZ}$ and $PF_X$, and a correspondence between the periods coming from the two mirror families. Both can be obtained from how we construct $\tX$ from $(X,L,f)$.

Moreover, our constructions and results can be used to estabilish a version of the open/closed correspondence based on Givental-style mirror symmetry. Given the semi-projective toric Calabi-Yau $4$-fold $\tZ$ as above, Givental's mirror theorem \cite{Givental98} equates the \emph{$J$-function} of $\tZ$, $J_{\tZ}$, defined via the closed Gromov-Witten invariants, to its \emph{$I$-function} $I_{\tZ}$, which is an explicit hypergeometric function that solves the Picard-Fuchs system $PF_{\tZ}$. On the open counterpart, for $(X,L,f)$, Fang-Liu \cite{FL13} define the \emph{extended} $J$-function $J_X$ using the disk invariants and show that it is equal to the \emph{extended} $I$-function $I_X$, which is an explicit solution to the extended Picard-Fuchs system $PF_X$. Our correspondences would allow us to directly relate the $J$-functions $J_X$ and $J_{\tZ}$. This would be further mirrored by a relation between the $I$-functions $I_X$ and $I_{\tZ}$ that can be derived from how we construct $\tX$ from $(X,L,f)$.


\subsubsection{Orbifolds and ineffective branes, open/closed correspondence for inner branes}\label{sect:InnerOrbi}
Our results can be extended to more general pairs $(X,L)$. First, we can allow $X$ to be an \emph{orbifold}, or a smooth toric Deligne-Mumford stack \cite{BCS05,FMN10} with trivial generic stablizer, and $L$ to be an \emph{ineffective} brane in $X$, or a brane with non-trivial generic stablizers. For this purpose, we may use the localization of orbifold Gromov-Witten invariants \cite{Liu13}, and that of open orbifold Gromov-Witten invariants \cite{FLT12,Ross14}.

The open/closed correspondence can also be generalized to an {\em inner} brane  $L$, that is, the unique torus invariant line in $X$ that it intersects is compact.

\subsubsection{Open-closed phase space and wall-crossing}  
As indicated by \cite{LM01,Mayr01}, the open/closed correspondence should be compatible with wall-crossings. Given a pair $(X,L)$, assuming that $X$ is semi-projective, we may change $X$ to a toric Calabi-Yau $3$-fold that differs from $X$ by a \emph{crepant transformation}, which is a birational map that preserves the canonical sheaf. We may also replace $L$ by a brane at a different location. For a pair $(X', L')$ arising in this way, if $\tX'$ is the associated toric Calabi-Yau $4$-fold and $\tZ'$ is a suitable semi-projective partial compactification, then $\tZ'$ should also differ from the semi-projective partial compactification $\tZ$ of $\tX$ by a crepant transformation. Such toric Calabi-Yau $4$-folds correspond to chambers in the \emph{secondary fan} $\Lambda$ of $\tZ$, and under the open/closed correspondence, these chambers correspond to the different possible pairs or \emph{open-closed phases} $(X',L')$. Crossing a wall in $\Lambda$ corresponds to a change of phase from either performing a crepant transformation on the $3$-fold, or moving the brane to an adjacent location. We may then view $\Lambda$ as the \emph{open-closed phase space} that parameterizes changes of phases. Our results in this work and their generalizations (Section \ref{sect:InnerOrbi}) would be the main ingredients in rigorously estabilishing this framework.

The correspondence between wall-crossings of $4$-folds and wall-crossings of open-closed phases has implications on its own. For instance, it provides an alternative approach to the \emph{Open Crepant Transformation Conjecture} for toric Calabi-Yau $3$-folds, which relates the disk invariants of open-closed phases related in the way above (see \cite{Yu20} and previous works summarized therein). Using the open/closed correspondence, one could give an alternative proof of this conjecture based on the well-established Crepant Transformation Conjecture that relates closed Gromov-Witten invariants of toric Calabi-Yau $4$-folds that differ by crepant transformations (see \cite{CIJ18} and previous works summarized therein).

\subsection{Organization of the paper}
In Section \ref{sect:Construction}, we construct the relative Calabi-Yau $3$-fold $(Y,D)$ and the local Calabi-Yau $4$-fold $\tX$ and discuss their basic geometric and topological relations to the toric Calabi-Yau $3$-fold $X$. Then in Section \ref{sect:GW}, we provide detailed localization computations of the open, relative, and closed Gromov-Witten invariants involved in our correspondences. Based on these computations, in Section \ref{sect:Correspondence}, we establish the open/relative (Theorem \ref{thm:OpenRelativeIntro}) and relative/local (Theorem \ref{thm:RelativeLocalIntro}) correspondences and thereby derive the open/closed correspondence (Theorem \ref{thm:OpenClosedIntro}).

\subsection{Acknowledgements}
We thank Pierrick Bousseau, Andrea Brini, Michel van Garrel, Helge Ruddat, and Fenglong You for illuminating discussions, comments, and suggestions.
The first named author is partially supported by NSF grant DMS-1564497.

\section{Toric geometry and constructions}\label{sect:Construction}
In this section, we give the explicit constructions of the relative Calabi-Yau 3-fold $(Y,D)$ and the local Calabi-Yau 4-fold $\tX$ from the data of a toric Calabi-Yau 3-fold $X$ with a framed Aganagic-Vafa outer brane $(L,f)$. We further compare the toric geometry and topology of the three spaces. In general, we use notations with hat ($\hat{\phantom{c}}$) while discussing $(Y,D)$ and its relative Gromov-Witten invariants, and notations with tilde ($\tilde{\phantom{c}}$) while discussing $\tX$ and its closed Gromov-Witten invariants.

\subsection{Notations}\label{sect:ToricNotations}
We work over $\bC$. We introduce some notations for a general $r$-dimensional smooth toric variety $Z$ specified by a fan $\Xi$ in $\bR^r$. We refer to \cite{CLS11, Fulton93} for the general theory of toric varieties.
\begin{itemize}
    \item For each $d= 0, \dots, r$, let $\Xi(d)$ be the set of $d$-dimensional cones in $\Xi$.

    \item For each $\sigma \in \Xi(d)$, let $V(\sigma) \subseteq Z$ denote the codimension-$d$ $(\bC^*)^r$-orbit closure in $Z$ corresponding to $\sigma$, and $i_\sigma: V(\sigma) \to Z$ denote the inclusion map. For $\sigma \in \Xi(r)$, denote $p_\sigma := V(\sigma)$ as the corresponding $(\bC^*)^r$-fixed point. For $\tau \in \Xi(r-1)$, denote $l_\tau:= V(\tau)$ as the corresponding $(\bC^*)^r$-invariant line.

    \item Define $\Xi(r-1)_c := \{\tau \in \Xi(r-1): l_\tau \mbox{ is compact}\}$. Moreover, define $Z^1_c := \bigcup_{\tau \in \Xi(r-1)_c} l_\tau$ as the toric 
    \emph{1-skeleton}\footnote{When $Z$ is non-compact, the toric 1-skeleton $Z^1_c$ in this paper is a subset of the 1-skeleton $Z^1=\bigcup_{\tau \in \Xi(r-1)} l_\tau$ defined in \cite[Section 4.2]{Liu13}.}  of $Z$.

    \item A \emph{flag} in $\Xi$ is a pair $(\tau, \sigma) \in \Xi(r-1) \times \Xi(r)$ where $\tau$ is a face of $\sigma$. Let $F(\Xi)$ denote the set of flags in $\Xi$.
\end{itemize}

\subsection{Open geometry of \texorpdfstring{$(X,L,f)$}{(X,L,f)}}\label{sect:NotationsX}
Let $N \cong \bZ^3$. Let $X$ be a smooth toric Calabi-Yau 3-fold, which is specified by a fan $\Sigma$ in $N_\bR := N \otimes \bR$. We introduce the following additional notations to describe the toric geometry of $X$:
\begin{itemize}
    \item Let $T:= N \otimes \bC^* \cong (\bC^*)^3$ be the algebraic torus contained in $X$ as an open dense subset. Let $M:= \Hom(N, \bZ)$, which can be canonically identified as the character lattice $\Hom(T, \bC^*)$ of $T$.

    \item Let $\Sigma(1) = \{\rho_1, \dots, \rho_R\}$. For each $i = 1, \dots, R$, let $b_i \in N$ be the primitive lattice vector on the ray $\rho_i$, i.e. $\rho_i \cap N = \bZ_{\ge 0}b_i$. Define
    $$
        \varphi: \bigoplus_{i = 1}^R \bZ e_i  \cong \bZ^R \to N, \qquad e_i \mapsto b_i.
    $$
    Then there is a short exact sequence of lattices
    \begin{equation}\label{eqn:FanSESX}
        \xymatrix{
            0 \ar[r] & \bL \ar[r]^\psi & \bZ^R \ar[r]^\varphi & N \ar[r] & 0.
        }
    \end{equation}
    Fix a $\bZ$-basis $\{\ep_1, \dots, \ep_{R-3}\}$ for $\bL:= \ker(\varphi) \cong \bZ^{R-3}$. For $a = 1, \dots, R-3$, define
    $$
        l^{(a)} = (l_1^{(a)}, \dots, l_R^{(a)}):= \psi(\ep_a) \in \bZ^{R},
    $$
    which are known as \emph{charge vectors} in the physics literature. They define a linear action of $G=(\bC^*)^{R-3}$ on $\bC^R$ by 
    \begin{equation}\label{eqn:G-action} 
    (s_1,\ldots,s_{R-3}) \cdot (x_1,\ldots,x_R) = \left(\prod_{a=1}^{R-3} s_a^{l_1^{(a)}} x_1,\ldots, \prod_{a=1}^{R-3} s_a^{l_R^{(a)}} x_R \right)
    \end{equation} 
    where $(s_1,\ldots,s_{R-3}) \in G$ and $(x_1,\ldots,x_R) \in \bC^R$. 

    \item For each cone $\sigma$ in $\Sigma$, define
        $$
            I_\sigma' = \{i \in \{1, \dots, R\} : \rho_i \subseteq \sigma\}, \qquad I_\sigma = \{1, \dots, R\} \setminus I_\sigma'.
        $$
        Define a monomial 
        $$
        \hat{x}_\sigma := \prod_{i\in I_\sigma} x_i \in \bC[x_1,\ldots,x_R].
        $$
        The set
        $$
        \cA(\Sigma) := \{ I_\sigma \mid \sigma\in \Sigma\}
        $$
        is the set of anti-cones determined by the fan $\Sigma$.
   \item  Let $B(\Sigma)$ be the ideal in $\bC[x_1,\ldots,x_R]$ generated by 
   $\{ \hat{x}_\sigma: \sigma \in \Sigma\}$, and let $Z(\Sigma)$ be the closed subvariety of 
   $$
   \bC^R =\mathrm{Spec} \text{ } \bC[x_1,\ldots,x_R]
   $$ 
   defined by the ideal $B(\Sigma)$.  Then $G$ acts freely on $\bC^R \setminus Z(\Sigma)$ and 
   \begin{equation}\label{eqn:quotient} 
   X = (\bC^R \setminus Z(\Sigma))/G. 
   \end{equation}
  \end{itemize}

\begin{assumption}\rm{
We assume that $X$ is \emph{Calabi-Yau}: the canonical bundle $K_X$ of $X$ is trivial. 
}\end{assumption}
This implies that there exists $u_3 \in M$ such that $\inner{u_3, b_i} = 1$ for all $i = 1, \dots, R$, where $\inner{-,-}: M \times N \to \bZ$ is the natural pairing. In other words, each $b_i$ belongs to the hyperplane $N' \times \{1\}$, where $N':=\ker(u_3) \cong \bZ^2$.

Let $P\subset N'_\bR \cong\bR^2$ be the cross section of the support $|\Sigma|$ of the fan in the hyperplane $N_\bR'\times\{1\}$. Then $|\Sigma|$ is a cone over $P$, and the
 fan $\Sigma$ determines a triangulation 
of $P$: the faces, edges, and vertices in the triangulation are in one-to-one correpondence with the 3-cones, 2-cones, and 1-cones in the fan $\Sigma$. 
In this paper we assume that all the maximal cones in $\Sigma$ are 3-dimensional, and that $P$ is a simple polygon, but we do not assume $P$ is convex.  If $P$ is convex then 
$X$ is semi-projective and \eqref{eqn:quotient} is a GIT quotient.  In general,  there is an toric open embedding
$X \hookrightarrow X'$ where $X'$ is a semi-projective smooth toric Calabi-Yau 3-fold defined by a fan $\Sigma'$ such that $\Sigma(1)=\Sigma'(1)$, $\Sigma \subset \Sigma'$. 
Indeed, since $\Sigma(1)=\Sigma'(1)$, $\Sigma$ and $\Sigma'$ determine the same charge vectors, and the open embedding $X\hookrightarrow X'$ is induced by the $G$-equivariant open embedding
 $\bC^R\setminus Z(\Sigma)  \hookrightarrow   \bC^R\setminus Z(\Sigma')$. (Note that $Z(\Sigma')\subset Z(\Sigma)$ since $\Sigma\subset \Sigma'$.)

 \begin{example} \rm{
$R=4$, 
 $b_1=(1,0,1)$, $b_2=(0,1,1)$, $b_3=(-1,-1,1)$, $b_4 = (0,0,1)$. Consider the two fans $\Sigma$ and $\Sigma'$ given in Figure \ref{fig:ExampleFans}. There is only one charge vector $(1,1,1,-3)\in \bZ^4$ which determines a $\bC^*$-action on $\bC^4$. 
 $$
 \begin{aligned}
 X_\Sigma =& \left( \bC^4 \setminus \{ x_1=x_2=0\}  \right) /\bC^* =\mathrm{Tot} \left( \cO_{\bP^1}(1)\oplus \cO_{\bP^2}(-3)\right),\\
 X_{\Sigma'} =& \left(\bC^4 \setminus \{ x_1=x_2= x_3=0\} \right)/\bC^* =\mathrm{Tot}\left(\cO_{\bP^2}(-3)\right).
 \end{aligned} 
 $$
 $X_{\Sigma'}$ is a semi-projective partial compactification of $X_{\Sigma}$ and is a GIT quotient.}
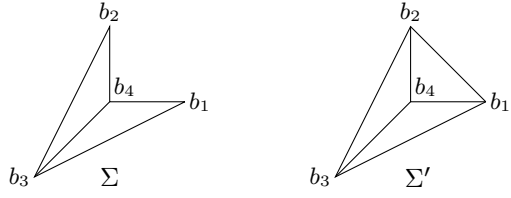
\begin{figure}[h]
 \begin{tikzpicture}
 \draw (1,0) -- (0,0) -- (0,1) -- (-1,-1) -- (1,0);
 \draw (0,0) -- (-1,-1);
 \draw (5,0) -- (4,1) -- (3,-1) -- (5,0); 
 \draw (5,0) -- (4,0) -- (4,1);
 \draw (4,0) -- (3,-1);  
 \node at (0,-1) {$\Sigma$};
 \node at (4.1,-1) {$\Sigma'$}; 

 \node at (1.2,0) {\small $b_1$};
 \node at (5.2,0) {\small $b_1$};
 \node at (0,1.2) {\small $b_2$};
 \node at (4,1.2) {\small $b_2$};
 \node at (-1.2, -1) {\small $b_3$}; 
 \node at (2.8,-1) {\small $b_3$}; 
 \node at (0.2,0.2) {\small $b_4$};
 \node at (4.2,0.2) {\small $b_4$};
 \end{tikzpicture}
 \caption{$\Sigma'$ contains $\Sigma$ as a subfan and its support is a cone over a convex polygon.}
\label{fig:ExampleFans}
 \end{figure} 

\end{example}

We now define the Aganagic-Vafa brane $L$. Let $X'$ be a semi-projective smooth toric Calabi-Yau 3-fold defined by a fan $\Sigma'$ with $\Sigma'(1) = \Sigma(1)$, which partially compactifies $X$. Then $X'$ is a GIT quotient $(\bC^R \setminus Z(\Sigma'))/G$. Moreover, $X'$ can be described as a symplectic quotient as follows: Let $G_{\bR} \cong U(1)^{R-3}$ be the maximal compact subgroup of $G$. The dual $\fg_{\bR}^*$ of the Lie algebra $\fg_{\bR}$ of $G_{\bR}$ can be canonically identified with $H^{1,1}(X'; \bR) \cong \bL_{\bR}^{\vee} \cong \bR^{R-3}$. The $G$-action \eqref{eqn:G-action} induces a Hamiltonian $G_{\bR}$-action on $\bC^R$, and let $\mu: \bC^{\bR} \to \fg_{\bR}^*$ be the moment map, which is given by
$$
    \mu(x_1, \dots, x_R) = \left(\sum_{i = 1}^R l_i^{(1)} |x_i|^2, \dots, \sum_{i = 1}^R l_i^{(R-3)} |x_i|^2 \right).
$$
Let $r = (r_1, \dots, r_{R-3}) \in H^{1,1}(X';\bR)$ be a K\"ahler class. Then
$$
    X' = \mu^{-1}(r)/G_{\bR},
$$
and the standard K\"ahler form
$$
    \frac{\sqrt{-1}}{2} \sum_{i = 1}^R dx_i \wedge d \bar{x}_i
$$
on $\bC^R$ descends to a K\"ahler form $\omega_r$ on $X'$ in the class $r$. Moreover, $\omega_r$ restricts to a symplectic form on $X$.

Write $x_i = |x_i|e^{\sqrt{-1}\phi_i}$. An Aganagic-Vafa brane \cite{AV00} $L$ in $X'$ is a Lagrangian submanifold defined by
$$
    \left\{ (x_1, \dots, x_R) \in \mu^{-1}(r) : \sum_{i = 1}^R l_i' |x_i|^2 = c', \sum_{i = 1}^R l_i'' |x_i|^2 = c'', \sum_{i = 1}^R \phi_i = \text{const}      \right\},
$$
where $c', c'' \in \bR$ are constants and $l' = (l_1', \dots, l_R'), l'' = (l_1'', \dots, l_R'') \in \bZ^{R}$ satisfy
$$
    \sum_{i = 1}^R l_i' = \sum_{i=1}^R l_i'' = 0.
$$
The submanifold $L$ is diffeomorphic to $S^1 \times \bC$ and intersects a unique $T$-invariant line $l_{\tau_0}$ in $X'$, where $\tau_0 \in \Sigma'(2)$. We make the following assumptions on $L$: 
\begin{assumption}\label{assump:AVBrane}
\rm{
With $L$ as above, we require that:
    \begin{itemize}
        \item $L \subset X$. This implies that $\tau_0 \in \Sigma(2)$.
        \item $\tau_0 \not \in \Sigma'(2)_c$, i.e. $L$ is an \emph{outer} brane in $X'$. This implies that $\tau_0 \not \in \Sigma(2)_c$ and $L$ is also an outer brane in $X$.
    \end{itemize}
}\end{assumption}

Let $\sigma_0 \in \Sigma(3)$ be the unique 3-cone that contains $\tau_0$. By permuting the indices if necessary, we assume that $I_{\sigma_0}' = \{1,2,3\}$, with $b_1, b_2, b_3$ appearing in $N' \times \{1\}$ in the counterclockwise order, and $I_{\tau_0}' = \{2,3\}$. Let $\tau_2, \tau_3 \in \Sigma(2)$ be the other two facets of $\sigma_0$ such that $I_{\tau_2}' = \{1,3\}$, $I_{\tau_3}' = \{1,2\}$. See Figure \ref{fig:AVBraneCones}.

\begin{figure}[h]
\begin{tikzpicture}
\draw (1,0) -- (0,1) -- (0,0) -- (1,0); 
\node[right] at (1,0) {\small $b_1$};
\node[above] at (0,1) {\small $b_2$};
\node at (-0.15,-0.15) {\small $b_3$};
\node at (0.3, 0.3) {\small $\sigma_0$};
\node[left] at (0,0.5) {\small $\tau_0$}; 
\node[below] at (0.5,0) {\small $\tau_2$};
\node at  (0.7,0.7) {\small $\tau_3$};
\end{tikzpicture}
\caption{Distinguished cones associated to the Aganagic-Vafa brane $L$.}
\label{fig:AVBraneCones}
\end{figure}
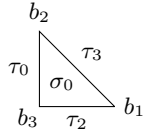

The flag $(\tau_0, \sigma_0)$ determines a $\bZ$-basis $\{v_1, v_2, v_3\}$ for $N$ under which
$$
    b_1 = (1,0,1), \qquad b_2 = (0,1,1), \qquad b_3 = (0,0,1).
$$
For $i = 1, \dots, R$, let $m_i, n_i \in \bZ$ such that $b_i = (m_i, n_i, 1)$ under the basis $\{v_1, v_2, v_3\}$. By Assumption \ref{assump:AVBrane}, $m_i \geq 0 $ for all $i$. We take $u_1, u_2 \in M$ such that $\{u_1, u_2, u_3\}$ is the $\bZ$-basis for $M$ dual to $\{v_1, v_2, v_3\}$. We have
$$
    H^*_T(\pt) = \bZ[\su_1, \su_2, \su_3].
$$
Let $T':= \ker(\su_3) \subset T$ be the 2-dimensional Calabi-Yau subtorus and $T'_\bR$ be the maximal compact subtorus of $T'$. Then $L$ is preserved under the $T'_\bR$-action.

Finally, let $f \in \bZ$ be a \emph{framing} on $L$. We will assume that the choice of $f$ is generic (see Assumption \ref{assump:GenFraming}).

\subsection{Relative geometry of \texorpdfstring{$(Y,D)$}{(Y,D)}}\label{sect:RelCons}
The smooth toric 3-fold $Y$ is defined by a fan $\hSi$ in $N_\bR$ specified as follows:
\begin{itemize}
    \item $\hSi(1) = \Si(1) \sqcup \{\rho_{R+1}\}$, where $\rho_{R+1} = \bR_{\ge 0} b_{R+1}$ and under the basis $\{v_1, v_2, v_3\}$ of $N$,
    $$
        b_{R+1} = (-1, -f, 0) \in N.
    $$
    Similar to above, for each cone $\hsi$ in $\hSi$, we define
        $$
            I_{\hsi}' = \{i \in \{1, \dots, R+1\} : \rho_i \subseteq \hsi\}, \qquad I_{\hsi} = \{1, \dots, R+1\} \setminus I_{\hsi}'.
        $$

    \item $\hSi(3) = \Si(3) \sqcup \{\hsi_0\}$, where $I_{\hsi_0}' = \{2,3, R+1\}$.

    \item $\hSi(2) = \Si(2) \cup \{\htau_2, \htau_3\}$, where $I_{\htau_2}' = \{3, R+1\}$, $I_{\htau_3}' = \{2, R+1\}$.
\end{itemize}
We define $D:= V(\rho_{R+1})$ as the toric divisor corresponding to the additional ray $\rho_{R+1}$. We have the following observations on the relation between $(Y,D)$ and $X$:
\begin{itemize}
    \item We have $Y \setminus D = X$. The pair $(Y,D)$ is log Calabi-Yau: $\bigwedge^3 \Omega_Y(\log D) \cong \cO_Y$.

    \item We define
    $$
        \hvphi: \bigoplus_{i = 1}^{R+1} \bZ e_i  \cong \bZ^{R+1} \to N, \qquad e_i \mapsto b_i,
    $$
    which restricts to $\varphi$ on $\bigoplus_{i = 1}^R \bZ e_i$. Then there is a short exact sequence of lattices
    \begin{equation}\label{eqn:FanSESY}
        \xymatrix{
            0 \ar[r] & \hbL \ar[r]^\hpsi & \bZ^{R+1} \ar[r]^\hvphi & N \ar[r] & 0.
        }
    \end{equation}
    We can view $\bL$ in \eqref{eqn:FanSESX} as a sublattice of $\hbL:= \ker(\hvphi) \cong \bZ^{R-2}$. We can pick $\ep_{R-2} \in \hbL$ such that $\{\ep_1, \dots, \ep_{R-2}\}$ is a $\bZ$-basis for $\hbL$ and we have
    $$
        \hl^{(a)} := \hpsi(\ep_a) = (l_1^{(a)}, \dots, l_R^{(a)},0) \qquad \mbox{for $a = 1,\dots, R-3$},
    $$
    $$
        \hl^{(R-2)}:= \hpsi(\ep_{R-2}) = (1, f, -f-1, 0, \dots, 0, 1).
    $$
    Similar to \eqref{eqn:quotient}, $Y$ can be expressed as a quotient $(\bC^{R+1} \setminus Z(\hSi))/(\bC^*)^{R-2}$ where the linear action of $(\bC^*)^{R-2}$ on $\bC^{R+1}$ is specified by the charge vectors $\hl^{(1)}, \dots, \hl^{(R-2)}$.
    
    \item $Y$ also contains the torus $T$ as an open dense subset.

    \item We have $\hSi(2)_c = \Sigma(2)_c \sqcup \{\tau_0\}$ and $Y^1_c = X^1_c \cup l_{\tau_0}$. Here $l_{\tau_0}$ is the (compact) $T$-orbit closure in $Y$. Moreover, $F(\hSi) = F(\Sigma) \sqcup \{(\tau_0, \hsi_0), (\htau_2, \hsi_0), (\htau_3, \hsi_0)\}$.
\end{itemize}

\subsection{Local geometry of \texorpdfstring{$\tX$}{X}}\label{sect:4FoldCons}
Let $\tN:= N \oplus \bZ v_4 \cong \bZ^4$ be a lattice that contains $N$ as a sublattice. The smooth toric Calabi-Yau 4-fold $\tX$ is defined by a fan $\tSi$ in $\tN_\bR:= \tN \otimes \bR$ specified as follows:
\begin{itemize}
    \item $\tSi(1) = \{\trho_1, \dots, \trho_{R+2}\}$, where $\trho_i = \bR_{\ge 0}\tb_i$ and under the basis $\{v_1, \dots, v_4\}$ of $\tN$,
    $$
        \tb_i = (m_i, n_i, 1, 0)\in \tN \qquad \mbox{for $i = 1, \dots, R$},
    $$
    $$
        \tb_{R+1} = (-1, -f, 1, 1), \qquad \tb_{R+2} = (0,0,1,1).
    $$
    Similar to above, for each cone $\tsi$ in $\tSi$, we define
        $$
            I_{\tsi}' = \{i \in \{1, \dots, R+2\} : \trho_i \subseteq \tsi\}, \qquad I_{\tsi} = \{1, \dots, R+1\} \setminus I_{\tsi}'.
        $$

    \item There is a bijection $\iota: \hSi(3) \to \tSi(4)$ such that for each $\hsi \in \hSi(3)$, 
    $$
        I_{\iota(\hsi)}' = I_{\hsi}' \sqcup \{R+2\}.
    $$
    \item $\tSi(2)$ and $\tSi(3)$ consists of faces of cones in $\tSi(4)$. Thus, similar to above, we have injective maps $\iota: \hSi(d) \to \tSi(d+1)$ for $d = 0, \dots, 3$ such that for each $\hsi \in \hSi(d)$,
    $$
        I_{\iota(\hsi)}' = I_{\hsi}' \sqcup \{R+2\}.
    $$
    Moreover, for $d = 1, 2, 3$, we have $\tSi(d) = \iota(\hSi(d-1)) \sqcup \hSi(d)$, where a cone $\hsi \in \hSi(d)$ can be viewed as a cone in $\tSi(d)$ via the inclusion $N \to \tN$.
\end{itemize}
We observe that $\tX$ is related to $X$ and $(Y,D)$ in the following ways:
\begin{itemize}
    \item $\tX$ is isomorphic to the total space of the line bundle
    $$
        \cO_Y(-V(\rho_1)- \cdots - V(\rho_{R})- D) \cong \cO_Y(-D)
    $$
    over $Y$, which restricts to $\cO_X$ over $X$. The projection $\tX \to Y$ is induced by the lattice map 
    $$
        \tN \to N, \qquad v_1 \mapsto v_1, \quad v_2 \mapsto v_2, \quad v_3 \mapsto v_3, \quad v_4 \mapsto -v_3
    $$
    whose kernel is generated by $\tb_{R+2}$. In particular, $\tX$ is Calabi-Yau.

    \item The inclusion $\iota: \hSi(2) \to \tSi(3)$ restricts to a bijection between $\hSi(2)_c$ and $\tSi(3)_c$, which induces an isomorphism $Y^1_c \cong \tX^1_c$.

    \item The injective maps $\iota$ induce an inclusion $\iota: F(\hSi) \to F(\tSi)$, $(\htau, \hsi) \mapsto (\iota(\htau), \iota(\hsi))$. In fact, we have
    \begin{equation}\label{eqn:FlagDecomp}
        F(\tSi) = \iota(F(\hSi)) \sqcup \{ (\hsi, \iota(\hsi)) : \hsi \in \hSi(3)\}.
    \end{equation}

    \item We define
    $$
        \tvphi: \bigoplus_{i = 1}^{R+2} \bZ e_i  \cong \bZ^{R+2} \to \tN, \qquad e_i \mapsto \tb_i,
    $$
    which restricts to $\varphi$ on $\bigoplus_{i = 1}^R \bZ e_i$. Then there is a short exact sequence of lattices
    $$
        \xymatrix{
            0 \ar[r] & \tbL \ar[r]^\tpsi & \bZ^{R+2} \ar[r]^\tvphi & \tN \ar[r] & 0.
        }
    $$
    We can identify $\tbL:= \ker(\tvphi) \cong \bZ^{R-2}$ with $\hbL$ in \eqref{eqn:FanSESY} in a way that
    $$
        \tl^{(a)} := \tpsi(\ep_a) = (l_1^{(a)}, \dots, l_R^{(a)},0,0) \qquad \mbox{for $a = 1,\dots, R-3$},
    $$
    $$
        \tl^{(R-2)}:= \tpsi(\ep_{R-2}) = (1, f, -f-1, 0, \dots, 0, 1,-1).
    $$
    Similar to \eqref{eqn:quotient}, $\tX$ can be expressed as a quotient $(\bC^{R+2} \setminus Z(\tSi))/(\bC^*)^{R-2}$ where the linear action of $(\bC^*)^{R-2}$ on $\bC^{R+2}$ is specified by the charge vectors $\tl^{(1)}, \dots, \tl^{(R-2)}$.

    \item Let $\tT:= \tN \otimes (\bC^*) \cong (\bC^*)^4$ be the algebraic torus contained in $\tX$ as an open dense subset. Let $\tM:= \Hom(\tN, \bZ)$, which can be canonically identified with the character lattice $\Hom(\tT, \bC^*)$ of $\tT$, and let $\{u_1, \dots, u_4\}$ be the $\bZ$-basis of $\tM$ dual to $\{v_1, \dots, v_4\}$. (We abuse notation here since $u_1, u_2, u_3 \in \tM$ maps to $u_1, u_2, u_3 \in M$ respectively under the natural map $\tM \to M$.) Then $T = \ker(u_4) \subset \tT$. We have
    $$
        H^*_{\tT}(\pt) = \bZ[\su_1, \su_2, \su_3, \su_4].
    $$
    The Calabi-Yau condition for $\tX$ is manifested in that $\inner{u_3, \tb_i} = 1$ for all $i = 1, \dots, R+2$, where $\inner{-,-}: \tM \times \tN \to \bZ$ is the natural pairing. Let $\tT' := \ker(u_3) \subset \tT$ be the 3-dimensional Calabi-Yau subtorus.
\end{itemize}

\subsection{Second homology groups}\label{sect:Homology}
With the choice of the Aganagic-Vafa brane $L$ in $X$, there is a split short exact sequence
$$
    \xymatrix{
        0 \ar[r] & H_2(X;\bZ) \ar[r] & H_2(X,L;\bZ) \ar[r]^\partial & H_1(L;\bZ) \ar[r] & 0.
    }
$$
The intersection $L \cap V(\tau_0)$ bounds a holomorphic disk $B$ in $V(\tau_0)$, which represents a class in $H_2(X,L;\bZ)$. We orient $B$ by the holomorphic structure of $X$. Then $H_1(L;\bZ) \cong \bZ \partial[B]$, and
$$
    H_2(X,L;\bZ) \cong H_2(X;\bZ) \oplus \bZ [B].
$$

Under the inclusion $\Sigma(2)_c \subset \hSi(2)_c$, there is an natural injection
$$
    H_2(X;\bZ) \to H_2(Y;\bZ), \qquad [l_{\tau}] \mapsto [l_{\tau}] \quad \mbox{for all } \tau \in \Sigma(2)_c,
$$
which can be extended to an isomorphism
$$
    \xymatrix{
        H_2(X,L;\bZ) \ar[r]^\cong & H_2(Y;\bZ)
    }
$$
if $[B]$ is mapped to the class of $l_{\tau_0}$ in $H_2(Y;\bZ)$. Moreover, under the bijection $\iota: \hSi(2)_c \to \tSi(3)_c$, the natural isomorphism
$$
    \xymatrix{
        H_2(Y;\bZ) \ar[r]^\cong & H_2(\tX;\bZ),
    }
    \qquad [l_{\htau}] \mapsto [l_{\iota(\htau)}] \quad \mbox{for all } \htau \in \hSi(2)_c,
$$
     identifies effective curve classes of $Y$ and $\tX$. If $\hbeta \in H_2(Y;\bZ)$ and $\tbeta \in H_2(\tX;\bZ)$ are corresponding effective classes, and $\tD:= V(\trho_{R+1})$, then
$$
    \hbeta \cdot D = \tbeta \cdot \tD.
$$

\section{Gromov-Witten invariants and localization computations}\label{sect:GW}
In this section, we give the definitions of the three types of Gromov-Witten invariants involved in our correspondences and compute them by virtual localization \cite{B97, GP99}. 
Details of localization computations of equivariant Gromov-Witten invariants of a general toric variety can be found in  \cite{Liu13}, which we follow closely. 

\subsection{Decorated graphs}\label{sect:DecGraphs}
We start by recalling the definition of decorated graphs, following \cite{Liu13}, which will index the components of the torus fixed locus of our moduli spaces of stable maps. We restrict our attention to the genus zero case. As in Section \ref{sect:ToricNotations}, let $Z$ be an $r$-dimensional smooth toric variety specified by a fan $\Xi$. 

\begin{definition} \rm{
Let $n \in \bZ_{\ge 0}$ and $\beta \in H_2(Z;\bZ)$ be an effective curve class. A genus-zero, $n$-pointed, degree-$\beta$ \emph{decorated graph} for $Z$ is a tuple $\vGa = (\Gamma, \vf, \vd, \vs)$, where:
\begin{itemize}
    \item $\Gamma$ is a compact, connected 1-dimensional CW complex. Let $V(\Gamma)$ denote the set of vertices in $\Gamma$, $E(\Gamma)$ denote the set of edges in $\Gamma$, and
    $$
        F(\Gamma):= \{(e,v) \in E(\Gamma) \times V(\Gamma): v \in e\}
    $$
    be the set of \emph{flags} in $\Gamma$.

    \item $\vf: V(\Gamma) \cup E(\Gamma) \to \Xi(r) \cup \Xi(r-1)_c$ is the \emph{label map} that sends each vertex $v \in V(\Gamma)$ to some $r$-dimensional cone $\sigma_v \in \Xi(r)$ and each edge $e \in E(\Gamma)$ to some $(r-1)$-dimensional cone $\tau_e \in \Xi(r-1)$. Moreover, $\vf$ induces a map $F(\Gamma) \to F(\Xi)$: for each $(e,v) \in F(\Gamma)$, we require that $(\tau_e, \sigma_v) \in F(\Xi)$.

    \item $\vd: E(\Gamma) \to \bZ_{>0}$ is the \emph{degree map}. We denote $d_e:= \vd(e)$ for each $e \in E(\Gamma)$.

    \item $\vs:\{1, \dots, n\} \to V(\Gamma)$ is the \emph{marking map}, defined if $n>0$.
\end{itemize}
such that the following are satisfied:
\begin{itemize}
    \item $\Gamma$ is a tree: $\displaystyle |E(\Gamma)| - |V(\Gamma)| +1 = 0$.
    \item $\displaystyle \sum_{e \in E(\Gamma)} d_e[l_{\tau_e}] = \beta$.
\end{itemize}
}\end{definition} 

Given $n \in \bZ_{\ge 0}$ and effective curve class $\beta \in H_2(Z;\bZ)$, let $\Gamma_{0,n}(Z, \beta)$ be the set of all genus-zero, $n$-pointed, degree-$\beta$ decorated graphs for $Z$. The set $\Gamma_{0,n}(Z, \beta)$ indexes the connected components of the $(\bC^*)^r$-fixed locus of $\Mbar_{0,n}(Z, \beta)$, the moduli space of genus-zero, degree-$\beta$ stable maps to $Z$ with $n$ marked points. For a stable map
$$
    u:(C, x_1, \dots, x_n) \to Z
$$
representing a point $[u] \in \Mbar_{0,n}(Z, \beta)^{(\bC^*)^r}$, the decorated graph $\vGa = (\Gamma, \vf, \vd, \vs)$ that indexes the connected component containing $[u]$ can be described as follows: 
\begin{itemize}
    \item The image of $u$ lies in $Z^1_c \subset Z$. 
   The set $V(\Gamma)$ of vertices in $\Gamma$ are in one-to-one correspondence with the set of connected components of $u^{-1}(Z^{(\bC^*)^r})$.
    Let $C_v$ denote the connected component associated to a vertex $v\in V(\Gamma)$. Then
    $u(C_v)$ is a $(\bC^*)^r$-fixed point $p_{\sigma_v}$, where $\sigma_v \in \Xi(r)$. We set $\vf(v) = \sigma_v$. 

    \item  The set $E(\Gamma)$ of edges in $\Gamma$ are in one-to-one correspondence with the set of irreducible components of $C$ that do not map constantly to $Z$. 
    Let $C_e$ denote the irreducible component associated to an edge $e \in E(\Gamma)$. Then  $C_e$ is isomorphic to $\bP^1$ and $u|_{C_e}$ is a degree-$d_e$ cover
    of $l_{\tau_e}\cong \bP^1$  fully ramified over the two $(\bC^*)^r$-fixed points in $l_{\tau_e}$, where $d_e\in \bZ_{>0}$ and $\tau_e\in \Xi(r-1)_c$. We set $\vf(e) = \tau_e$ and $\vd(e)=d_e$. 
    
    \item $F(\Gamma)$ is the set of pairs $(e,v)\in E(\Gamma)\times V(\Gamma)$ such that $C_e \cap C_v$ is non-empty.  

    \item For each marked point $x_i$, we set  $\vs(i)=v$ if $x_i\in C_v$. 
\end{itemize}

Given a decorated graph $\vGa \in \Gamma_{0,n}(Z, \beta)$, we introduce the following notations:
\begin{itemize} 
    \item For each $v \in V(\Gamma)$, define $E_v := \{e \in E(\Gamma): (e,v) \in F(\Gamma)\}$ and $S_v := \vs^{-1}(v)$. Define $\val(v) := |E_v|$ and $n_v := |S_v|$.

    \item Let
    $$
        V^S(\Gamma):= \{ v \in V(\Gamma): \val(v) - 2 + n_v >0\}
    $$ 
    be the set of \emph{stable} vertices. The unstable vertices are partitioned into three subsets:
    \begin{align*}
        V^1(\Gamma) &:= \{v \in V(\Gamma): \val(v) = 1, n_v = 0\}\\
        V^{1,1}(\Gamma) &:= \{v \in V(\Gamma): \val(v) = 1, n_v = 1\}\\
        V^2(\Gamma) &:= \{v \in V(\Gamma): \val(v) = 2, n_v = 0\}
    \end{align*}

    \item Let $\Aut(\vGa)$ denote the group of \emph{automorphisms} of $\vGa$, i.e. automorphisms of the graph $\Gamma$ that make the label maps $\vf$, $\vd$, $\vs$ invariant.
\end{itemize}
The connected component of $\Mbar_{0,n}(Z, \beta)^{(\bC^*)^r}$ indexed by $\vGa$ can be identified as a quotient $[\cM_{\vGa}/A_{\vGa}]$. Here, 
$$
    \cM_{\vGa} := \prod_{v \in V^S(\Gamma)} \Mbar_{0, E_v \cup S_v}
$$
where $\Mbar_{0, E_v \cup S_v}$ is the moduli space $\Mbar_{0, \val(v)+n_v}$ with the marked points labeled by $E_v \cup S_v$; $A_{\vGa}$ is the automorphism group of a point in $\cF_{\vGa}$, which fits into an exact sequence
$$
    1 \to \prod_{e \in E(\Gamma)} \bZ_{d_e} \to A_{\vGa} \to \Aut(\vGa) \to 1.
$$

\begin{remark}\rm{
For our toric Calabi-Yau 3-fold $X$ (resp. 4-fold $\tX$), we will consider the action of the Calabi-Yau subtorus $T' \subset T$ (resp. $\tT' \subset \tT$) on the moduli spaces of stable maps. Since $X^{T'} = X^T$ (resp. $\tX^{\tT'} = \tX^{\tT}$), the $T'$-fixed (resp. $\tT'$-fixed) loci of the moduli spaces of stable maps can be identified with the $T$-fixed (resp. $\tT$-fixed) loci, and can thus be characterized by decorated graphs in the same way as above.
}\end{remark}

\subsection{Weights of Calabi-Yau tori}
We compare the weights of the actions of the Calabi-Yau tori $T'$ and $\tT'$ at the fixed points in $X$, $Y$, and $\tX$. We have $T' = \ker(u_4) \subset \tT'$, and
$$
    H_{T'}^*(\pt) = \bZ[\su_1, \su_2], \qquad H_{\tT'}^*(\pt) = \bZ[\su_1, \su_2, \su_4].
$$
Given $(\htau, \hsi) \in F(\hSi)$, we define
$$
    \bw(\htau,\hsi) := c_1^{T'}(T_{p_{\hsi}}l_{\htau}) \in H^2_{T'}(\pt),
$$
which is the weight of the $T'$-action on the tangent line $T_{p_{\hsi}}l_{\htau}$ to $l_{\htau}$ at the fixed point $p_{\hsi}$ in $Y$. In the case $(\htau, \hsi) \in F(\Sigma)$, that is $\hsi \neq \hsi_0$, $\bw(\htau, \hsi)$ is also the weight of the $T'$-action on the tangent line $T_{p_{\hsi}}l_{\htau}$ to $l_{\htau}$ at the fixed point $p_{\hsi}$ in $X$. By our choice of the bases $\{v_1, v_2, v_3\}, \{u_1, u_2, u_3\}$ based on the flag $(\tau_0, \sigma_0)$ (see Section \ref{sect:NotationsX}), we have
\begin{equation}\label{eqn:Sigma0Wts}
    \begin{aligned}
        &\bw(\tau_0, \sigma_0)  = \su_1, &&\bw(\tau_2, \sigma_0) = \su_2, && \bw(\tau_3, \sigma_0) = -\su_1 - \su_2,\\
        &\bw(\tau_0, \hsi_0) = -\su_1, &&\bw(\htau_2, \hsi_0) = -f\su_1 + \su_2, && \bw(\htau_3, \hsi_0) = f\su_1 - \su_2.
    \end{aligned}
\end{equation}
See Figure \ref{fig:Sigma0Wts}. In particular, the normal bundle $N_{l_{\tau_0}/Y}$ is isomorphic to $\cO(f) \oplus \cO(-f-1)$.

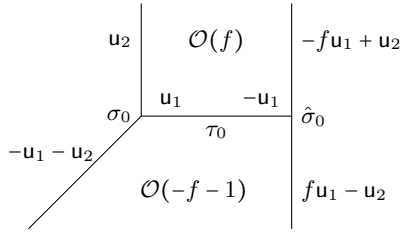
\begin{figure}[h]
\begin{tikzpicture}
\draw (-2.5,-1.5) -- (-1,0) -- (-1,1.5); 
\draw (-1,0) -- (1,0);
\draw (1, 1.5) -- (1, -1.5);
\node[left] at (-1,0) {\small $\sigma_0$};
\node[below] at (0,0) {\small $\tau_0$};
\node[right] at (1,0) {\small $\hsi_0$};
\node[right] at (1,-1) {\small $f\su_1-\su_2$};
\node[right] at (1,1) {\small $-f\su_1+\su_2$};
\node[left] at (-1,1) {\small $\su_2$};
\node[left] at (-1.5, -0.5) {\small $-\su_1 - \su_2$};
\node[above] at (-0.6, 0) {\small $\su_1$};
\node[above] at (0.6, 0) {\small $-\su_1$};
\node at (0, 1) {\small $\cO(f)$};
\node at (-0.3, -1) {\small $\cO(-f-1)$};
\end{tikzpicture}
\caption{$T'$-weights of the distinguished flags, marked on the toric graph.}
\label{fig:Sigma0Wts}
\end{figure}

Similarly, given $(\ttau, \tsi) \in F(\tSi)$, we define
$$
    \tbw(\ttau,\tsi) := c_1^{\tT'}(T_{p_{\tsi}}l_{\ttau}) \in H^2_{\tT'}(\pt),
$$
which is the weight of the $\tT'$-action on the tangent line $T_{p_{\tsi}}l_{\ttau}$ to $l_{\ttau}$ at the fixed point $p_{\tsi}$ in $\tX$. If $(\ttau, \tsi) = \iota(\htau, \hsi)$ for some $(\htau, \hsi) \in F(\hSi)$, then
\begin{equation}\label{eqn:TanWtsEqual}
    \tbw(\ttau, \tsi) \big|_{\su_4 = 0} = \bw(\htau, \hsi).
\end{equation}
(Here we identify $H^*_{T'}(\pt)$ with the subring $\bZ[\su_1, \su_2]$ of $H^*_{\tT'}(\pt)$.) Otherwise, $(\ttau, \tsi) = (\hsi, \iota(\hsi))$ for some $\hsi \in \hSi(3)$ (see \eqref{eqn:FlagDecomp}). In this case, we have
\begin{equation}\label{eqn:Sigma0TanWts}
    \tbw(\ttau, \tsi) = \tbw(\hsi, \iota(\hsi)) = \begin{cases}
        \su_4 & \mbox{if } \hsi \neq \hsi_0\\
        \su_1 + \su_4 & \mbox{if } \hsi = \hsi_0.
    \end{cases}
\end{equation}
We summarize the $\tT'$-weights associated to the cones $\iota(\sigma_0)$ and $\iota(\hsi_0)$ as follows:
\begin{equation}\label{eqn:tSigma0Wts}
    \begin{aligned}
        &\tbw(\iota(\tau_0), \iota(\sigma_0))  = \su_1, &&\tbw(\iota(\tau_2), \iota(\sigma_0)) = \su_2, && \tbw(\iota(\tau_3), \iota(\sigma_0)) = -\su_1 - \su_2-\su_4, && \tbw(\sigma_0, \iota(\sigma_0)) = \su_4,\\
        &\tbw(\iota(\tau_0), \iota(\hsi_0)) = -\su_1, && \tbw(\iota(\htau_2), \iota(\hsi_0)) = -f\su_1 + \su_2, && \tbw(\iota(\htau_3), \iota(\hsi_0)) = f\su_1-\su_2-\su_4, && \tbw(\hsi_0, \iota(\hsi_0)) = \su_1 + \su_4.
    \end{aligned}
\end{equation}
In particular, 
$$
N_{l_{\iota(\tau_0)}/\tX}\cong N_{l_{\tau_0}/Y} \oplus \cO(-1) \cong \cO(f) \oplus \cO(-f-1) \oplus \cO(-1). 
$$
\begin{assumption}\label{assump:GenFraming} \rm{
In this work, we make the assumption that the choice of the framing $f$ is \emph{generic} (with respect to  the effective class $\beta$), i.e. avoiding a finite set of $\bZ$. }
\end{assumption}

\subsection{Disk invariants of \texorpdfstring{$(X,L,f)$}{(X,L,f)}} \label{sec:disk-invariants}
In this section, we recall the definition and computation of disk invariants of $(X,L,f)$ from \cite[Section 3.5]{FL13}. See also \cite{KL01, FLT12}.

\subsubsection{Definition}
Let $\beta' \in H_2(X;\bZ)$ be an effective curve class and $d \in \bZ_{>0}$. Let
$$
    \Mbar(X, L \mid \beta', d)
$$
be the moduli space of morphisms
$$
    u:(C, \partial C) \to (X,L)
$$
where
\begin{itemize}
    \item $C$ is a connected prestable bordered Riemann surface of arithmetic genus $0$ and a single boundary component $\partial C$;

    \item $u_*[C] = \beta' + d[B]$ and $u_*[\partial C] = d\partial[B]$ (see Section \ref{sect:Homology} for the definition of the disk $B$);

    \item The automorphism group of $u$ is finite. Here, an automorphism of $u$ is an automorphism of $(C, \partial C)$ that makes $u$ invariant.
\end{itemize}
$\Mbar(X, L \mid \beta', d)$ is a possibly singular orbifold with corners, equipped with a virtual tangent bundle which is a virtual real vector bundle of rank zero. 
The action of the compact Calabi-Yau torus $T'_\bR \cong U(1)^2$ on $(X, L)$ induces an action on $\Mbar(X, L \mid \beta', d)$
which lifts to an action on its virtual tangent bundle. The fixed locus $\Mbar(X,L\mid\beta', d)^{T'_{\bR}}$ is proper, and each of its connected component
is a compact complex orbifold.

The class-$(\beta', d)$ \emph{disk invariant} of $(X, L,f)$ is defined by
$$
    N_{\beta', d}^{X,L,f}:= \int_{[\Mbar(X, L \mid \beta', d)^{T'_{\bR}}]^\vir}\frac{1}{e_{T'_{\bR}}(N^\vir)} \bigg|_{\su_2 - f\su_1 = 0} \in \bQ
$$
where $N^\vir$ is the virtual normal bundle of $\Mbar(X, L \mid \beta', d)^{T'_{\bR}}$ in $\Mbar(X,L\mid\beta',d)$. Since $f$ is generic (Assumption \ref{assump:GenFraming}), $N_{\beta', d}^{X,L,f}$ is well-defined.

The disk invariants of $(X,L,f)$ are a special class of \emph{open Gromov-Witten invariants} which virtually count stable maps whose domain can have possibly higher genus and more than one boundary components. The open Gromov-Witten invariants of $(X,L,f)$ are first defined through the \emph{formal} relative Gromov-Witten invariants of $(\hat{Y}, \hat{D})$ introduced by Li-Liu-Liu-Zhou \cite{LLLZ09}, where $(\hat{Y}, \hat{D})$ is the formal completion of $(Y,D)$ along the 1-skeleton $Y^1_c$. Fang-Liu \cite{FL13} showed that this definition is equivalent to the one we give above, at least in the case of disk invariants, and in fact that the disk invariants are well-defined for any framing $f \in \bZ$.

\subsubsection{Localization computations} \label{sect:OpenLocalize}
Given a stable map $u:(C, \partial C) \to (X,L)$ that represents a point in $\Mbar(X, L \mid \beta', d)^{T'_{\bR}}$, we can write
$$
    C = C' \cup C'',
$$
where $C'$ is a closed Riemann surface of arithmetic genus $0$ and $C''$ is a disk attached to $C'$ at a node $q$ of $C$ with $\partial C'' = \partial C$. We have that $u|_{C'}:(C', q) \to X$ represents a point in $\Mbar_{0,1}(X, \beta')^{T'}$ and $u|_{C''}: (C'', \partial C) \to (B, B \cap L)$ is a degree-$d$ cover. In particular, $u$ maps $q$ to $p_{\sigma_0}$. Therefore, the connected components of $\Mbar(X, L \mid \beta', d)^{T'_{\bR}}$ are indexed by the decorated graphs
$$
    \Gamma_{0, 1}^0(X, \beta') := \{ \vGa \in \Gamma_{0, 1}(X, \beta') : \vf \circ \vs(1) = \sigma_0\}.
$$
By comparing the perfect obstruction theories of $\Mbar(X, L \mid \beta', d)$ and $\Mbar_{0,1}(X, \beta')$, the proof of Proposition 3.4 of Fang-Liu \cite{FL13} computes $N^{X,L,f}_{\beta', d}$ as follows:

\begin{theorem}[Fang-Liu \cite{FL13}]
For any effective class $\beta' \in H_2(X;\bZ)$ of $X$ and $d \in \bZ_{>0}$,
\begin{equation}\label{eqn:FLDiskComputation}
    N^{X,L,f}_{\beta', d} = (-1)^{fd}\frac{\prod_{k = 1}^{d-1}(fd+k)}{d \cdot d!} \cdot \int_{[\Mbar_{0, 1}(X,\beta')^{T'}]^\vir} \frac{1}{e_{T'}(N^{\vir})}\cdot \frac{\ev^*_1\phi}{(\frac{\su_1}{d})(\frac{\su_1}{d} - \psi_1)}  \bigg|_{\su_2 - f\su_1 = 0},
\end{equation}
where $N^\vir$ is the virtual normal bundle of the $T'$-fixed locus of $\Mbar_{0,1}(X, \beta')$, and $\phi \in H^6_{T'}(X;\bZ)$ is the $T'$-equivariant Poincar\'e dual of the fixed point $p_{\sigma_0}$.
\end{theorem}

Using the computation of $e_{T'}(N^\vir)$ on components of $\Mbar_{0,1}(X,\beta')^{T'}$ (see \cite[Theorem 73]{Liu13}), we can rewrite \eqref{eqn:FLDiskComputation} as
\begin{align*}
        N^{X,L,f}_{\beta', d} =& (-1)^{fd}\frac{\prod_{k = 1}^{d-1}(fd+k)}{d \cdot d!}
        \cdot  \sum_{\vGa \in \Gamma_{0,1}(X, \beta')} \prod_{e \in E(\Gamma)} \frac{\bh(\tau_e, d_e)}{d_e} \prod_{v \in V(\Gamma)} (\bw(\sigma_v)^{\val(v)-1}(i_{\sigma_v}^*\phi)^{n_v})\\
        &\cdot  \prod_{v \in V(\Gamma)} \int_{\Mbar_{0, E_v \cup S_v}} \frac{1}{((\frac{\su_1}{d})(\frac{\su_1}{d} - \psi_1))^{n_v} \cdot \prod_{e \in E_v}(\frac{\bw(\tau_e, \sigma_v)}{d_e} - \psi_{(e,v)})} \bigg|_{\su_2 - f\su_1 = 0}.
\end{align*}
Here, in the integral, $\psi_{(e,v)}$ is the $\psi$-class of $\Mbar_{0, E_v \cup S_v}$ corresponding to the marked point $e \in E_v$. Moreover, for each $\vGa \in \Gamma_{0,1}(X, \beta')$ there is a unique vertex $v \in V(\Gamma)$ with $n_v \neq 0$, for which $n_v = 1$ and $S_v = \{1\}$. We adopt the following convention for unstable vertices:
\begin{equation}\label{eqn:UnstableIntegral}
    \int_{\Mbar_{0,1}} \frac{1}{\sw_1 - \psi_1} = \sw_1, \qquad \int_{\Mbar_{0,2}} \frac{1}{(\sw_1 - \psi_1)(\sw_2 - \psi_2)} = \frac{1}{\sw_1+\sw_2}, \qquad \int_{\Mbar_{0,2}} \frac{1}{\sw_1 - \psi_1} = 1.
\end{equation}
The quantities involved in the equation above are defined as follows:
\begin{itemize}
    \item For each $\tau \in \Sigma(2)_c$ and $d \in \bZ_{>0}$, if $f_d: \bP^1 \to l_{\tau}$ denotes the degree-$d$ cover totally ramified over the two $T$-fixed points of $l_{\tau}$, we define
    \begin{equation}\label{eqn:LocalPieceDefHTau}
        \bh(\tau, d) := \frac{e_{T'}(H^1(\bP^1, f_d^*TX)^m)}{e_{T'}(H^0(\bP^1, f_d^*TX)^m)},
    \end{equation}
    where the superscript ``$m$" stands for the \emph{moving part}:Any complex representation $V$ of $T'$ is a direct sum of $1$-dimensional representations and can be written as a direct sum $V=V^f\oplus V^m$, where the \emph{fixed part} $V^f$ (resp. \emph{moving part} $V^m$) is a direct sum of trivial (resp. non-trivial) $1$-dimensional $T'$-representations.
    
    \item For each $\sigma \in \Sigma(3)$, we define
    \begin{equation}\label{eqn:LocalPieceDefWSigma}
        \bw(\sigma) := e_{T'}(T_{p_{\sigma}}X) = \prod_{(\tau, \sigma) \in F(\Sigma)} \bw(\tau,\sigma).
    \end{equation}

\end{itemize}

To further simplify the above computation of $N^{X,L,f}_{\beta',d}$, observe that
$$
    i_{\sigma}^*\phi = \begin{cases}
            0 & \mbox{if } \sigma \neq \sigma_0\\
            \bw(\sigma_0) & \mbox{if } \sigma = \sigma_0.
        \end{cases}
$$
Therefore only decorated graphs in $\Gamma_{0,1}^0(X, \beta')$ contribute. We have
\begin{proposition}\label{prop:OpenLocalResult}
For any effective class $\beta' \in H_2(X;\bZ)$ of $X$ and $d \in \bZ_{>0}$,
\begin{align*}      
        N^{X,L,f}_{\beta', d} =& (-1)^{fd}\frac{\prod_{k = 1}^{d-1}(fd+k)}{ d!\su_1} \cdot  \sum_{\vGa \in \Gamma_{0, 1}^0(X, \beta')} \frac{1}{|\Aut(\vGa)|} \cdot \prod_{e \in E(\Gamma)} \frac{\bh(\tau_e, d_e)}{d_e} \\
        & \cdot  \prod_{v \in V(\Gamma)} \int_{\Mbar_{0, E_v \cup S_v}} \frac{\bw(\sigma_v)^{\val(v)-1+n_v}}{ (\frac{\su_1}{d} - \psi_1)^{n_v} \cdot \prod_{e \in E_v}(\frac{\bw(\tau_e, \sigma_v)}{d_e} - \psi_{(e,v)})} \bigg|_{\su_2 - f\su_1 = 0}.
    \end{align*}
\end{proposition}

\subsection{Maximally tangent relative Gromov-Witten invariants of \texorpdfstring{$(Y,D)$}{(Y,D)}}
In this section, we compute the genus-zero maximally-tangent relative Gromov-Witten invariants of the log Calabi-Yau 3-fold $(Y,D)$, as a special class of relative Gromov-Witten invariants. As our computation will show, these invariants can be recovered from the \emph{formal} relative Gromov-Witten invariants of $(\hat{Y}, \hat{D})$ introduced by Li-Liu-Liu-Zhou \cite{LLLZ09}, where $(\hat{Y}, \hat{D})$ is the formal completion of $(Y,D)$ along the 1-skeleton $Y^1_c$. We use the moduli spaces of relative stable maps defined by Li \cite{Li01, Li02}. 
Our computation is similar to those in \cite{LLLZ09, LLZ03}.

\subsubsection{Expanded targets and torus action}\label{sect:RelTarget}
Let $\Delta(D)$ be the total space of the projective line bundle
$$
    \bP(\cO_D \oplus N_{D/Y}) \to D.
$$
The action of the Calabi-Yau torus $T'$ on $D$ extends to an action on $\Delta(D)$, under which the fixed locus is the fiber over $p_{\hsi_0}$.

For each $m \in \bZ_{\ge 0}$, define
$$
    Y[m] = Y \cup \Delta_{(1)} \cup \cdots \cup \Delta_{(m)},
$$
where each $\Delta_{(i)}$ is a copy of $\Delta(D)$ with distinguised sections $D_{(i-1)} = \bP(\cO_D \oplus 0)$ and $D_{(i)} = \bP(0 \oplus N_{D/Y})$, and
$$
    Y \cap \Delta_{(1)} = D_{(0)}, \qquad \Delta_{(i)} \cap \Delta_{(i+1)} = D_{(i)} \quad \mbox{for } i = 1, \dots, m-1.
$$
In particular, for $m=0$, we have $(Y[0], D_{(0)}) = (Y,D)$. We denote
$$
    Y(m) = \Delta_{(1)} \cup \cdots \cup \Delta_{(m)},
$$
which admits a map to $D_{(0)} = D$ by projection. This induces a projection map
$$
    \pi_m: Y[m] \to Y.
$$
We denote $p_{(0)} = p_{\hsi_0}$. For each $i = 1, \dots, m$, let $l_{(i)}:= \pi_m^{-1}(p_{(0)}) \cap \Delta_{(i)} \cong \bP^1$ and $p_{(i)}$ be the unique point where $l_{(i)}$ and $D_{(i)}$ intersects. Let
$$
    Y^1_c(m) := l_{(1)} \cup \cdots \cup l_{(m)},
$$
which is a chain of $m$ copies of $\bP^1$'s, and $Y^1_c[m]:= Y^1_c \cup Y^1_c(m)$. Then the torus $T'$ acts on $Y(m)$ with fixed locus $Y^1_c(m)$ and the projection $\pi_m:Y[m] \to Y$ is $T'$-equivariant. In addition, there is a $(\bC^*)^m$-action on $Y(m)$ that scales the fiber direction of each $\Delta_{(i)}$ and makes the projection $Y(m) \to D$ invariant. This action extends to an action on $Y[m]$ that pointwise fixes $Y$ and makes $\pi_m$ invariant, which restricts to an action on $Y^1_c[m]$ that pointwise fixes $Y^1_c$.

\subsubsection{Definition}
Let $\hbeta \in H_2(Y;\bZ)$ be an effective curve class of $Y$ such that $d := \hbeta \cdot D>0$. Let
$$
    \Mbar(Y/D, \hbeta)
$$
be the moduli space of morphisms
$$
    u: (C, x) \to (Y[m], D_{(m)})
$$
where
\begin{itemize}
    \item $(C, x)$ is a connected prestable Riemann surface of arithmetic genus $0$ with a single marked point $x$;

    \item $m \in \bZ_{\ge 0}$ and $(Y[m],D_{(m)})$ is as defined in Section \ref{sect:RelTarget};

    \item $(\pi_m \circ u)_*[C] = \hbeta$.

    \item $u(x) \in D_{(m)}$ and $u^{-1}(D_{(m)}) = dx$ as Cartier divisors.

    \item For each $i = 1, \dots, m-1$, the preimage $u^{-1}(D_{(i)})$ consists of nodes in $C$. If $q \in u^{-1}(D_{(i)})$ and $C_1, C_2$ are the two irreducible components of $C$ that intersect at $q$, then $u|_{C_1}$ and $u|_{C_2}$ have the same contact order to $D_{(i)}$ at $q$.

    \item The automorphism group of $u$ is finite. Here, an automorphism of $u$ is a pair $(\alpha_1, \alpha_2)$, where $\alpha_1$ is an automorphism of $(C,x)$ and $\alpha_2$ is an automorphism of $(Y[m], D_{(m)})$ that makes $\pi_m$ invariant, such that $u \circ \alpha_1 = \alpha_2 \circ u$.
\end{itemize}
$\Mbar(Y/D, \hbeta)$ is a Deligne-Mumford stack with a perfect obstruction theory of virtual dimension $1$. The action of $T'$ on the expanded targets 
$(Y[m], D_{(m)})$ induces an action on $\Mbar(Y/D, \hbeta)$ under which the fixed locus is proper.

The genus-zero, degree-$\hbeta$ \emph{maximally-tangent relative Gromov-Witten invariant} of $(Y,D)$ is defined by
$$
    N_{\hbeta}^{Y,D}:= \int_{[\Mbar(Y/D, \hbeta)^{T'}]^\vir}\frac{\su_2 - f\su_1}{e_{T'}(N^\vir)} \bigg|_{\su_2 - f\su_1 = 0} \in \bQ
$$
where $N^\vir$ is the virtual normal bundle of $\Mbar(Y/D, \hbeta)^{T'}$. Since $f$ is generic (Assumption \ref{assump:GenFraming}), as our computation will show, on each component $\cF$ of $\Mbar(Y/D, \hbeta)^{T'}$, the power of $\su_2 - f\su_1$ in the expression $\displaystyle \int_{[\cF]^\vir}\frac{\su_2 - f\su_1}{e_{T'}(N^\vir)}$ is always non-negative. This implies that $N_{\hbeta}^{Y,D}$ is well-defined.

\begin{remark} \rm{
Our choice of the class $\su_2 - f\su_1$ in the integration follows the choice by Li-Liu-Liu-Zhou \cite{LLLZ09}: The moduli space $\Mbar(Y/D, \hbeta)$ admits an evaluation map $\ev$ to the divisor $D$ associated to the marked point $x$. Set $L := l_{\htau_2} \subset D$. Then our choice of class is $\ev^* c_1^{T'}(\cO_D(L)) = \su_2 - f\su_1$.
}\end{remark}

\begin{remark} \rm{
The weight restriction $\su_2 - f\su_1 = 0$ will turn out to be unnecessary: By our localization computation in the rest of this section, the expression
$$
    \int_{[\Mbar(Y/D, \hbeta)^{T'}]^\vir}\frac{\su_2 - f\su_1}{e_{T'}(N^\vir)}
$$
can be identified with a formal relative Gromov-Witten invariant of $(\hat{Y}, \hat{D})$ introduced by Li-Liu-Liu-Zhou \cite{LLLZ09}, which is shown to be a rational number independent of $\su_1, \su_2$ (see \cite[Theorem 4.8]{LLLZ09}).
}\end{remark}

For the rest of this section, we compute $N_{\hbeta}^{Y,D}$ by localization. The end result is given in Proposition \ref{prop:RelativeLocalResult}. We first describe connected components of the $T'$-fixed locus of $\Mbar(Y/D, \hbeta)$ in terms of decorated graphs, and then compute the contribution of each component by computing the virtual normal bundle from the moving part of the perfect obstruction theory.

\subsubsection{Fixed locus of moduli}\label{sect:RelFixed}
Let $u:(C, x) \to (Y[m], D_{(m)})$ be a relative stable map that represents a point $[u] \in \Mbar(Y/D, \hbeta)^{T'}$. Let $\tu:= \pi_m \circ u: C \to Y$ denote the composition. The image of $u$ lies in $Y^1_c[m] \subset Y[m]$, and the image of $\tu$ lies in $Y^1_c \subset Y$. Then we associate to $u$ the decorated graph $\vGa = (\Gamma, \vf, \vd, \vs) \in \Gamma_{0, 1}(Y, \hbeta)$ associated to $\tu$ as in Section \ref{sect:DecGraphs}. Note that $\vf \circ \vs(1)=\hsi_0$. Let $\Gamma^{Y,D}_{\hbeta}$ denote the set of all decorated graphs that arise this way, which indexes the connected components of $\Mbar(Y/D, \hbeta)^{T'}$.

Given $\vGa \in \Gamma^{Y,D}_{\hbeta}$, let $\cF_{\vGa}$ denote the connected component of $\Mbar(Y/D, \hbeta)^{T'}$ indexed by $\vGa$. We partition $V(\Gamma)$ into two subsets
$$
    V(\Gamma)^{(0)} := \{v \in V(\Gamma): \vf(v) \neq \hsi_0 \}, \qquad V(\Gamma)^{(1)} :=\vf^{-1}(\hsi_0).
$$
Note that for any $v \in V(\Gamma)^{(0)}$, $n_v = 0$ and $S_v = \emptyset$. Moreover, $V(\Gamma)^{(1)} = \{\hv_0(\vGa)\}$ is a singleton set. Let $\mu(\vGa) = (\mu(\vGa)_1, \dots, \mu(\vGa)_{\ell(\mu(\vGa))})$ be the partition of $d$ determined by the degrees of $\tu$ restricted to the components in $\vf^{-1}(\tau_0) \subset E(\Gamma)$. Note that $\ell(\mu(\vGa)) =1$ if and only if $m=0$. We consider two cases separately:
\begin{itemize}
    \item {\it Case I:} $\ell(\mu(\vGa))=1$, i.e. $\mu(\vGa) = (d)$. In this case, $\vGa$ satisfies the following:

    \begin{itemize}
        \item[$\circ$] $C_{\hv_0(\vGa)} = \{x\}$.
    
        \item[$\circ$] There is a unique edge $e_0(\vGa) \in E(\Gamma)$ such that $\vf(e_0(\vGa)) = \tau_0$. We have $(e_0(\vGa), \hv_0(\vGa)) \in F(\Gamma)$ and $\vd(e_0(\vGa)) = d$.

        \item[$\circ$] $\vs(1) = \hv_0(\vGa)$, $n_{\hv_0(\vGa)} = 1$, and $S_{\hv_0(\vGa)} = \{1\}$.
    \end{itemize}

    There is a map $i_{\vGa}: \cM_{\vGa} \to \Mbar(Y/D, \hbeta)$ with image $\cF_{\vGa}$, under which $\cF_{\vGa}$ can be identified as the quotient $[\cM_{\vGa}/A_{\vGa}]$ as in Section \ref{sect:DecGraphs}.

    \item {\it Case II:} $\ell(\mu(\vGa)) > 1$. In this case, for each relative stable map $u: C \to Y[m]$ whose associated graph is $\vGa$, the restriction to $C_{\hv_0(\vGa)}$ represents a point in 
$$
\Mbar_{\vGa}^{(1)}: =\Mbar_{0, 0}(\bP^1, \mu(\vGa), (d)) \sslash \bC^*,
$$ 
the moduli space of relative stable maps to  the non-rigid $(\bP^1,0,\infty)$ with relative condition $\mu(\vGa)$ at $0$ and $(d)$ at $\infty$. Such a map has form
    $$
        u': (C', y_1, \dots, y_{\ell(\mu(\vGa))}, x) \to \bP^1(m)
    $$
    where:
    \begin{itemize}
        \item[$\circ$] $(C', y_1, \dots, y_{\ell(\mu(\vGa))}, x)$ is a connected prestable Riemann surface of arithmetic genus $0$ with $\ell(\mu(\vGa))+1$ marked points.

        \item[$\circ$] $u'(x) = p_{(m)}$, and $(u')^{-1}(p_{(m)}) = dx$ as Cartier divisors; $u'(y_{j}) = p_{(0)}$ for each $j$, and $(u')^{-1}(p_{(0)}) = \mu(\vGa)_1y_{1} + \cdots + \mu(\vGa)_{\ell(\mu(\vGa))}y_{\ell(\mu(\vGa))}$ as Cartier divisors.

        \item[$\circ$] For $i = 1, \dots, m-1$, the preimage $(u')^{-1}(p_{(i)})$ consists of nodes in $C'$. If $q \in (u')^{-1}(p_{(i)})$ and $C_1, C_2$ are the two irreducible components of $C'$ that intersect at $q$, then $u'|_{C_1}$ and $u'|_{C_2}$ have the same contact order to $p_{(i)}$ at $q$.

        \item[$\circ$] The automorphism group of $u'$ is finite. Here, an automorphism of $u'$ is a pair $(\alpha_1, \alpha_2)$, where $\alpha_1$ is an automorphism of $C'$ fixing $x$ and each $y_j$ and $\alpha_2\in (\bC^*)^m$ is an automorphism of $(\bP^1(m), p_{(0)}, p_{(m)})$, such that $u' \circ \alpha_1 = \alpha_2 \circ u'$.
    \end{itemize}

    There is a map $i_{\vGa}$ from 
    $$
        \hat{\cM}_{\vGa} := \left(\prod_{v \in V^S(\Gamma) \cap V(\Gamma)^{(0)}} \Mbar_{0, E_v} \right) \times \Mbar_{\vGa}^{(1)}
    $$
    to $\Mbar(Y/D, \hbeta)$ with image $\cF_{\vGa}$, under which $\cF_{\vGa}$ can be identified as a quotient $[\hat{\cM}_{\vGa}/\hA_{\vGa}]$, where
    $$
        \hA_{\vGa} = \{ \alpha \in A_{\vGa} : \alpha(e) = e \mbox{ for each } e \in \vf^{-1}(\tau_0)\}.
    $$
\end{itemize}
We denote
$$
    \Gamma^{Y,D,0}_{\hbeta} := \{\vGa \in \Gamma^{Y,D}_{\hbeta} : \ell(\mu(\vGa)) = 1\}.
$$

\subsubsection{Virtual normal bundle}
At a point $[u: (C, x) \to (Y[m], D_{(m)})] \in \Mbar(Y/D, \hbeta)$, the tangent space $T^1$ and the obstruction space $T^2$ are determined by the following two exact sequences
of complex vector spaces:
$$
    0 \to \Ext^0(\Omega_C(x), \cO_C) \to  H^0(\bfD^\bullet) \to T^1 \to \Ext^1(\Omega_C(x), \cO_C) \to H^1(\bfD^\bullet) \to T^2 \to 0,
$$
\begin{align*} 
        0 \to & H^0(C, u^*\Omega_{Y[m]}(\log D_{(m)}))^\vee \to H^0(\bfD^\bullet) \to \bigoplus_{i = 0}^{m-1}H^0_{\et}(\bfR_i^\bullet) \to\\
        \to & H^1(C, u^*\Omega_{Y[m]}(\log D_{(m)}))^\vee \to H^1(\bfD^\bullet) \to \bigoplus_{i = 0}^{m-1}H^1_{\et}(\bfR_i^\bullet) \to 0.
\end{align*}
Here, if $\{q_1, \dots, q_{n_i}\}$ is a list of nodes in $u^{-1}(D_{(i)})$, then
$$
    H^0_{\et}(\bfR_i^\bullet) \cong \bigoplus_{j = 1}^{n_i} T_{q_j}(u^{-1}(\Delta_{(i)})) \otimes T_{q_j}^*(u^{-1}(\Delta_{(i)})); 
$$
$$
    H^1_{\et}(\bfR_i^\bullet) \cong \left( \bigoplus_{j = 1}^{n_i} (u|_{\{q_j\}})^{-1}N_{D_{(i)}/\Delta_{(i-1)}} \otimes N_{D_{(i)}/\Delta_{(i)}}\right) \slash \bC
$$
where as all the 1-dimensional vector spaces $(u|_{\{q_j\}})^{-1}N_{D_{(i)}/\Delta_{(i-1)}} \otimes N_{D_{(i)}/\Delta_{(i)}}$ are isomorphic, we mod out the direct sum by the diagonal embedding of this vector space.

The obstruction theory of $\Mbar(Y/D, \hbeta)$ is $T'$-equivariant, and the virtual normal bundle $N^{\vir}$ of $\Mbar(Y/D, \hbeta)^{T'}$ is given by the moving part. From a calculation similar to that in \cite[Appendix A]{LLZ03}, we obtain the following: For a decorated graph $\vGa \in \Gamma^{Y,D,0}_{\hbeta}$, we have
$$
    i^*_{\vGa} \left(\frac{1}{e_{T'}(N^\vir)}\right) = \prod_{v \in V(\Gamma)^{(0)}} \frac{\bw(\sigma_v)^{\val(v)-1}}{\prod_{e \in E_v} (\frac{\bw(\tau_e, \sigma_v)}{d_e} - \psi_{(e,v)})} \cdot \prod_{e \in E(\Gamma)} \bh(\tau_e, d_e),
$$
where:
\begin{itemize}
    \item For each $\htau \in \hSi(2)_c$ and $d \in \bZ_{>0}$, if $f_d: \bP^1 \to l_{\htau}$ denotes the degree $d$ cover totally ramified over the two $T$-fixed points of $l_{\htau}$, we define
    $$
        \bh(\htau, d) := \frac{e_{T'}(H^1(\bP^1, f_d^*\Omega_Y(\log D)^\vee)^m)}{e_{T'}(H^0(\bP^1, f_d^*\Omega_Y(\log D)^\vee)^m)}.
    $$
    If $\htau \in \Sigma(2)_c$, the above definition is consistent with \eqref{eqn:LocalPieceDefHTau} in Section \ref{sect:OpenLocalize} since $(\Omega_Y(\log D)^\vee)|_{X} = TX$.

    \item For each $\hsi \in \hSi(3)$, we define
    $$
        \bw(\hsi) := e_{T'}(T_{p_{\hsi}}Y) = \prod_{(\htau, \hsi) \in F(\hSi)} \bw(\htau,\hsi).
    $$
    If $\hsi \in \Sigma(3)$, the above definition is consistent with \eqref{eqn:LocalPieceDefWSigma}.
\end{itemize}

Moreover, for $\vGa \in \Gamma^{Y,D}_{\hbeta} \setminus \Gamma^{Y,D,0}_{\hbeta}$,  
\begin{equation}\label{eqn:RelVirNormalPositive}
\begin{aligned}
    i^*_{\vGa} \left(\frac{1}{e_{T'}(N^\vir)}\right)  =&   (-\su_1- \psi^t)^{\ell(\mu(\vGa))-1} \cdot  \frac{  \left( (-f\su_1+\su_2)(f\su_1-\su_2) \right)^{\ell(\mu(\vGa))-1} }
    {\prod_{j=1}^{\ell(\mu(\vGa))} (-\frac{\su_1}{\mu(\vGa)_j} - \psi_{(e,\hv_0(\vGa))} )}  \\ 
    & \cdot \prod_{v \in V(\Gamma)^{(0)} } \frac{\bw(\sigma_v)^{\val(v)-1}}{\prod_{e \in E_v} (\frac{\bw(\tau_e, \sigma_v)}{d_e} - \psi_{(e,v)})} \cdot \prod_{e \in E(\Gamma)} \bh(\tau_e, d_e),
 \end{aligned} 
\end{equation}
where $\psi^t$ is the \emph{target psi class} of $\Mbar_{\vGa}^{(1)}= \Mbar_{0,0}(\bP^1, \mu(\vGa), (d))\sslash \bC^*$ at $0 \in \bP^1$ (see e.g. \cite[Section 5]{LLZ07} for the definition). We note in particular that the term $(-\su_1- \psi^t)^{\ell(\mu(\vGa))-1}$ comes from the moving part of $\bigoplus_{i = 0}^{m-1}H^1_{\et}(\bfR_i^\bullet)$ in the perfect obstruction theory. For each $e \in E_{\hv_0(\vGa)}$, we have
$$
    \frac{  -\su_1 - \psi^t}{-\frac{\su_1}{d_e} - \psi_{(e,\hv_0(\vGa))}} =   d_e.
$$
Thus \eqref{eqn:RelVirNormalPositive} can be simplified as follows:
\begin{eqnarray*}
    i^*_{\vGa} \left(\frac{1}{e_{T'}(N^\vir)}\right) & = & \prod_{j = 1}^{\ell(\mu(\vGa))}\mu(\vGa)_j \cdot  (-1)^{\ell(\mu(\vGa)) - 1}(\su_2-f\su_1)^{2\ell(\mu(\vGa))-2} 
   \cdot \frac{1}{-\su_1 - \psi^t}  \\
   & & \cdot \prod_{v \in V(\Gamma)^{(0)}} \frac{\bw(\sigma_v)^{\val(v)-1}}{\prod_{e \in E_v} (\frac{\bw(\tau_e, \sigma_v)}{d_e} - \psi_{(e,v)})} \cdot \prod_{e \in E(\Gamma)} \bh(\tau_e, d_e).
\end{eqnarray*}

\subsubsection{Summary of computation}
We summarize our computation of the maximally-tangent relative Gromov-Witten invariants of $(Y,D)$ as follows:
\begin{proposition}\label{prop:RelativeLocalResult}
Let $\hbeta \in H_2(Y;\bZ)$ be an effective curve class of $Y$ such that $d := \hbeta \cdot D>0$. Then
    \begin{align*}      
        N^{Y,D}_{\hbeta} =& \sum_{\vGa \in \Gamma^{Y,D,0}_{\hbeta}} \frac{\su_2-f\su_1}{|\Aut(\vGa)|} \cdot \prod_{e \in E(\Gamma)} \frac{\bh(\tau_e, d_e)}{d_e} 
         \cdot  \prod_{v \in V(\Gamma)^{(0)}} \int_{\Mbar_{0, E_v}} \frac{\bw(\sigma_v)^{\val(v)-1}}{\prod_{e \in E_v} (\frac{\bw(\tau_e, \sigma_v)}{d_e} - \psi_{(e,v)})} \bigg|_{\su_2 -f\su_1 = 0}\\
        +& \sum_{\vGa \in \Gamma^{Y,D}_{\hbeta} \setminus \Gamma^{Y,D,0}_{\hbeta}} \frac{\su_2-f\su_1}{|\Aut(\vGa)|} \cdot  \prod_{e \in E(\Gamma)} \frac{\bh(\tau_e, d_e)}{d_e} 
         \cdot \prod_{v \in V(\Gamma)^{(0)}} \int_{\Mbar_{0, E_v }} \frac{\bw(\sigma_v)^{\val(v)-1}}{\prod_{e \in E_v} (\frac{\bw(\tau_e, \sigma_v)}{d_e} - \psi_{(e,{}v)})}\\
        & \cdot   \prod_{j = 1}^{\ell(\mu(\vGa))}\mu(\vGa)_j  \cdot (-1)^{\ell(\mu(\vGa))-1} (\su_2-f\su_1)^{2 \ell(\mu(\vGa))-2}
        \int_{\Mbar_{\vGa}^{(1)}} \frac{1}{-\su_1- \psi^t} \bigg|_{\su_2 -f\su_1 = 0}.
    \end{align*}
\end{proposition}
Here, we adopt the integration convention \eqref{eqn:UnstableIntegral} as in Section \ref{sect:OpenLocalize}. In fact, we will show later that there is no contribution from decorated graphs in $\Gamma^{Y,D}_{\hbeta} \setminus \Gamma^{Y,D,0}_{\hbeta}$ after the weight restriction $\su_2 -f\su_1 = 0$ (see Lemma \ref{lem:OpenRelativePositive}). The dimension of $\Mbar_{\vGa}^{(1)}$ is $\ell(\mu(\vGa)) -2$, and
$$
\int_{\Mbar_{\vGa}^{(1)}} \frac{1}{-\su_1- \psi^t} =  (-\su_1)^{1-\ell(\mu(\vGa))} \int_{\Mbar_{\vGa}^{(1)}} (\psi^t)^{\ell(\mu(\vGa))-2}. 
$$

\subsection{Closed Gromov-Witten invariants of \texorpdfstring{$\tX$}{X}}
In this section, we use localization to compute the closed Gromov-Witten invariants of the toric Calabi-Yau 4-fold $\tX$, following \cite{Liu13}. We restrict our attention to genus zero.

\subsubsection{Definition}
Let $\tbeta \in H_2(\tX;\bZ)$ be a non-zero effective curve class of $\tX$. Let $\Mbar_{0,0}(\tX,\tbeta)$ (resp. $\Mbar_{0,1}(\tX,\tbeta)$) be the moduli space of genus-zero, $0$-pointed (resp. $1$-pointed), degree-$\tbeta$ stable maps to $\tX$, which is a Deligne-Mumford stack with a perfect obstruction theory of virtual dimension $1$ (resp. $2$). The action of the Calabi-Yau 3-torus $\tT'$ on $\tX$ induces a $\tT'$-action on $\Mbar_{0,0}(\tX,\tbeta)$ (resp. $\Mbar_{0,1}(\tX,\tbeta)$) under which the fixed locus is proper.

We consider two genus-zero, degree-$\tbeta$ \emph{closed} \emph{Gromov-Witten invariants} of $\tX$. First, we define a $0$-pointed invariant
\begin{equation}\label{eqn:zero-point}
    N_{\tbeta,0}^{\tX} := \int_{[\Mbar_{0,0}(\tX,\tbeta)^{\tT'}]^\vir} \frac{\su_2-f\su_1}{e_{\tT'}(N^\vir)}  \bigg|_{\su_4 = 0, \su_2 - f\su_1 = 0} \in \bQ,
\end{equation}
where $N^\vir$ is the virtual normal bundle of $\Mbar_{0,0}(\tX,\tbeta)^{\tT'}$. Moreover, we define a $1$-pointed invariant
\begin{equation}\label{eqn:one-point} 
    N_{\tbeta,1}^{\tX} := \int_{[\Mbar_{0,1}(\tX,\tbeta)^{\tT'}]^\vir} \frac{\ev_1^*([\tD][\tD_2])}{e_{\tT'}(N^\vir)}  \bigg|_{\su_4 = 0, \su_2 - f\su_1 = 0} \in \bQ,
\end{equation} 
where $\tD = V[\trho_{R+1}]$, $\tD_2 := V[\trho_2]$, and $N^\vir$ is the virtual normal bundle of $\Mbar_{0,1}(\tX,\tbeta)^{\tT'}$.  
Since $f$ is generic (Assumption \ref{assump:GenFraming}), as our computation will show, on each component $\cF$ of $\Mbar_{0,0}(\tX,\tbeta)^{\tT'}$ (resp. $\Mbar_{0,1}(\tX,\tbeta)^{\tT'}$), the powers of $\su_2 - f\su_1$ and of $\su_4$ in the expression $\displaystyle \int_{[\cF]^\vir} \frac{\su_2-f\su_1}{e_{\tT'}(N^\vir)}$ $\displaystyle \left(\text{resp. } \int_{[\cF]^\vir} \frac{\ev_1^*([\tD][\tD_2])}{e_{\tT'}(N^\vir)} \right)$ are always non-negative. This implies that $N_{\tbeta, 0}^{\tX}$ (resp. $N_{\tbeta, 1}^{\tX}$) is well-defined.

We now make two observations. First, by \eqref{eqn:tSigma0Wts}, we have the following identity in $H^*_{\tT'}(\tX;\bQ)$: 
\begin{equation}\label{eqn:DD}
[\tD] [\tD_2] =  (\su_2-f\su_1)[\tD].
\end{equation}
Second, the evaluation map $\ev_1$ is $\tT'$-equivariant, so 
$$
\ev_1^*: H^*_{\tT'}(\tX;\bQ) \lra H^*_{\tT'}(\Mbar_{0,1}(\tX,\tbeta)^{\tT'} ;\bQ)
$$
is a morphism of modules over $H^*(B\tT';\bQ) =\bQ[\su_1,\su_2,\su_4]$. In particular, we have the following
identity in $H^*_{\tT'}(\Mbar_{0,1}(\tX,\tbeta);\bQ)$:
\begin{equation} \label{eqn:linear}
\ev_1^*((\su_2-f\su_1)[\tD]) = (\su_2-f\su_1)\ev_1^* [\tD].
\end{equation}
It follows from  \eqref{eqn:one-point}, \eqref{eqn:DD}, and \eqref{eqn:linear} that
\begin{equation}\label{eqn:divisor} 
    N_{\tbeta,1}^{\tX} = \int_{[\Mbar_{0,1}(\tX,\tbeta)^{\tT'}]^\vir} \frac{ (\su_2-f\su_1) \ev_1^*[\tD]}{e_{\tT'}(N^\vir)}  \bigg|_{\su_4 = 0, \su_2 - f\su_1 = 0}. 
\end{equation} 
Comparing the right hand sides of \eqref{eqn:divisor} and \eqref{eqn:zero-point}, we see that $N_{\tbeta,1}^{\tX}$ has a divisor insertion $[\tD]$ while $N_{\tbeta, 0}^{\tX}$ has none. In view of the divisor equation, one expects the following relation:
\begin{lemma}\label{lem:DivisorAxiom}
For any effective class $\tbeta \in H_2(\tX;\bZ)$ of $\tX$ satisfying $d:= \tbeta \cdot \tD >0$, we have
$$
    N_{\tbeta,1}^{\tX} = d N_{\tbeta,0}^{\tX}.
$$
\end{lemma}
We defer the proof to Section \ref{sect:RelLocal}.

\subsubsection{Localization computations}
Components of the fixed locus $\Mbar_{0,0}(\tX,\tbeta)^{\tT'}$ (resp. $\Mbar_{0,1}(\tX,\tbeta)^{\tT'}$) are indexed by decorated graphs in $\Gamma_{0,0}(\tX, \tbeta)$ (resp. $\Gamma_{0,1}(\tX, \tbeta)$) as defined in Section \ref{sect:DecGraphs}. By \cite[Theorem 73]{Liu13}, we can compute the closed Gromov-Witten invariants as follows:
\begin{proposition}\label{prop:ClosedLocalResult}
For any non-zero effective curve class $\tbeta \in H_2(\tX;\bZ)$ of $\tX$, we have
\begin{equation}\label{eqn:ClosedLocalResult0}
    N_{\tbeta,0}^{\tX} = \sum_{\vGa \in \Gamma_{0,0}(\tX, \tbeta)} \frac{\su_2-f\su_1}{|\Aut(\vGa)|} \cdot \prod_{e \in E(\Gamma)} \frac{\tbh(\ttau_e, d_e)}{d_e}
        \cdot \prod_{v \in V(\Gamma)} \int_{\Mbar_{0, E_v}} \frac{\tbw(\tsi_v)^{\val(v)-1}}{\prod_{e \in E_v}(\frac{\tbw(\ttau_e,\tsi_v)}{d_e} - \psi_{(e,v)})} \bigg|_{\su_4 = 0, \su_2 - f\su_1 = 0},
\end{equation}
\begin{equation}\label{eqn:ClosedLocalResult1}
        N_{\tbeta,1}^{\tX} = \sum_{\vGa \in \Gamma_{0,1}(\tX, \tbeta)} \frac{\su_2-f\su_1}{|\Aut(\vGa)|} \cdot \prod_{e \in E(\Gamma)} \frac{\tbh(\ttau_e, d_e)}{d_e}
        \cdot \prod_{v \in V(\Gamma)} \int_{\Mbar_{0, E_v \cup S_v}} \frac{\tbw(\tsi_v)^{\val(v)-1}i_{\tsi_v}^*([\tD])^{n_v}}{\prod_{e \in E_v}(\frac{\tbw(\ttau_e,\tsi_v)}{d_e} - \psi_{(e,v)})} \bigg|_{\su_4 = 0, \su_2 - f\su_1 = 0}.
\end{equation}
\end{proposition}
Here, we adopt the integration convention \eqref{eqn:UnstableIntegral} as in Section \ref{sect:OpenLocalize}. The quantities involved are defined as follows:
\begin{itemize}
    \item For each $\ttau \in \tSi(3)_c$ and $d \in \bZ_{>0}$, if $f_d: \bP^1 \to l_{\ttau}$ denotes the degree-$d$ cover totally ramified over the two $\tT$-fixed points of $l_{\ttau}$, we define
$$
    \tbh(\ttau, d) := \frac{e_{\tT'}(H^1(\bP^1, f_d^*T\tX)^m)}{e_{\tT'}(H^0(\bP^1, f_d^*T\tX)^m)}.
$$

    \item For each $\tsi \in \tSi(4)$, we define
$$
    \tbw(\tsi) := e_{\tT'}(T_{p_{\tsi}}\tX) = \prod_{(\ttau, \tsi) \in F(\tSi)} \tbw(\ttau,\tsi).
$$
\end{itemize}

\section{Correspondences}\label{sect:Correspondence}
Based on the localization computations of the Gromov-Witten invariants in the previous section, we establish the open/closed correspondence (Theorem \ref{thm:OpenClosedIntro}) in this section, given as Theorem \ref{thm:OpenClosed} in Section \ref{sect:OpenClosed}. This is obtained as a composition of the open/relative correspondence (Theorem \ref{thm:OpenRelativeIntro}), given as Theorem \ref{thm:OpenRelative} in Section \ref{sect:OpenRel}, and the relative/local correspondence (Theorem \ref{thm:RelativeLocalIntro}), given as Theorem \ref{thm:RelativeLocal} in Section \ref{sect:RelLocal}.

\subsection{Open/relative correspondence}\label{sect:OpenRel}
In this section, we identify the disk invariants of $(X,L,f)$ and the genus-zero maximally-tangent relative Gromov-Witten invariants of $(Y,D)$. Recall from Section \ref{sect:Homology} that there is an isomorphism $H_2(X,L;\bZ) \cong H_2(Y;\bZ)$ that identifies the class of the holomorphic disk $B$ with the class of $l_{\tau_0}$.

\begin{theorem}[Open/relative correspondence \cite{FL13, LLLZ09}]\label{thm:OpenRelative}
Let $\beta' \in H_2(X;\bZ)$ be an effective class of $X$, $d \in \bZ_{>0}$, and $\hbeta \in H_2(Y;\bZ)$ be the effective class of $Y$ corresponding to $
\beta:= \beta' + d[B]$ under the isomorphism $H_2(X,L;\bZ) \cong H_2(Y;\bZ)$. Then
$$
    N^{X,L,f}_{\beta', d} = (-1)^{d + 1}N^{Y,D}_{\hbeta}.
$$
\end{theorem}

\begin{remark}\rm{
Theorem \ref{thm:OpenRelative} is a special case of a general correspondence between open Gromov-Witten invariants of $(X,L,f)$ and relative Gromov-Witten invariants of $(Y,D)$, which involves invariants of higher genus and general winding/tangency profiles. The general open/relative correspondence can be obtained from Fang-Liu \cite{FL13} (see Proposition 3.4 in the outer brane case) or Li-Liu-Liu-Zhou \cite{LLLZ09}, under the identification of the relative Gromov-Witten invariants of $(Y,D)$ and the formal relative Gromov-Witten invariants of $(\hat{Y}, \hat{D})$.
}\end{remark}

Based on Propositions \ref{prop:OpenLocalResult} and \ref{prop:RelativeLocalResult}, we establish Theorem \ref{thm:OpenRelative} in two steps: First, in Lemma \ref{lem:OpenRelativeZero}, we identify $N^{X,L,f}_{\beta', d}$ with the contribution to $N^{Y,D}_{\hbeta}$ from stable maps with an unexpanded target $(Y,D)$. Second, in Lemma \ref{lem:OpenRelativePositive}, we show that there is no contribution from stable maps with an expanded target $(Y[m],D[m]), m>0$ after the weight restriction $\su_2 - f\su_1 = 0$, as mentioned in the discussion after Proposition \ref{prop:RelativeLocalResult}.

For $\hbeta$ as in Theorem \ref{thm:OpenRelative}, and any $\vGa \in \Gamma^{Y,D,0}_{\hbeta}$, we set
\begin{equation}\label{eqn:RelContribution}
    \hC_{\vGa}:= \frac{\su_2-f\su_1}{|\Aut(\vGa)|} \cdot \prod_{e \in E(\Gamma)} \frac{\bh(\tau_e, d_e)}{d_e} 
         \cdot  \prod_{v \in V(\Gamma)^{(0)}} \int_{\Mbar_{0, E_v}} \frac{\bw(\sigma_v)^{\val(v)-1}}{\prod_{e \in E_v} (\frac{\bw(\tau_e, \sigma_v)}{d_e} - \psi_{(e,v)})}.
\end{equation}
to be the contribution of $\vGa$ to $N^{Y,D}_{\hbeta}$ before the weight restriction $\su_2-f\su_1=0$; similarly, for any $\vGa \in \Gamma^{Y,D}_{\hbeta} \setminus \Gamma^{Y,D,0}_{\hbeta}$, we set
\begin{align*}      
    \hC_{\vGa}: = & \frac{\su_2-f\su_1}{|\Aut(\vGa)|} \cdot  \prod_{e \in E(\Gamma)} \frac{\bh(\tau_e, d_e)}{d_e} 
         \cdot \prod_{v \in V(\Gamma)^{(0)}} \int_{\Mbar_{0, E_v }} \frac{\bw(\sigma_v)^{\val(v)-1}}{\prod_{e \in E_v} (\frac{\bw(\tau_e, \sigma_v)}{d_e} - \psi_{(e,{}v)})}\\
        & \cdot   \prod_{j = 1}^{\ell(\mu(\vGa))}\mu(\vGa)_j  \cdot (-1)^{\ell(\mu(\vGa))-1} (\su_2-f\su_1)^{2 \ell(\mu(\vGa))-2}
        \int_{\Mbar_{\vGa}^{(1)}} \frac{1}{-\su_1- \psi^t} 
\end{align*}
to be the contribution of $\vGa$ before the weight restriction.

\begin{lemma}\label{lem:OpenRelativeZero}
For $\beta', d, \hbeta$ as in Theorem \ref{thm:OpenRelative}, we have
$$    
    N^{X,L,f}_{\beta', d} = (-1)^{d + 1} \sum_{\vGa \in \Gamma^{Y,D,0}_{\hbeta}} \hC_{\vGa} \big|_{\su_2 - f\su_1=0}.
$$
\end{lemma}

\begin{proof}
First, we show that there is a natural bijection between the sets of decorated graphs
$$
    \Gamma_{0, 1}^0(X, \beta') \longleftrightarrow \Gamma^{Y,D,0}_{\hbeta}.
$$
Given $\vGa' \in \Gamma_{0,1}^0(X, \beta')$, we may obtain a decorated graph $\vGa \in \Gamma^{Y,D,0}_{\hbeta}$ by replacing the marked point of $\vGa'$ by a new vertex $\hv_0$ with label $\vf(\hv_0) = \hsi_0$, a new edge $e_0$ connecting $\hv_0$ to $\vs(1)$ with degree $\vd(e_0) = d$ and label $\vf(e_0) = \tau_0$, and a new marked point $1$ with $\vs(1) = \hv_0$. Conversely, given $\vGa \in \Gamma^{Y,D,0}_{\hbeta}$, we may obtain a decorated graph $\vGa' \in \Gamma_{0,1}^0(X, \beta')$ by removing $\hv_0(\vGa)$ and $e_0(\vGa)$ and moving the marked point to the place of the unique flag in $\vf^{-1}(\tau_0, \sigma_0)$. 

Now let $\vGa' \in \Gamma_{0, 1}^0(X, \beta')$ and $\vGa \in \Gamma^{Y,D,0}_{\hbeta}$ be a pair of corresponding decorated graphs. Note that $\Aut(\vGa') = \Aut(\vGa)$. We compare the contribution of $\vGa'$ to $N^{X,L,f}_{\beta', d}$ in Proposition \ref{prop:OpenLocalResult}, which is
$$
    (-1)^{fd}\frac{\prod_{k = 1}^{d-1}(fd+k)}{ d!\su_1 \cdot |\Aut(\vGa')|} \cdot \prod_{e \in E(\Gamma')} \frac{\bh(\tau_e, d_e)}{d_e} \cdot  \prod_{v \in V(\Gamma')} \int_{\Mbar_{0, E_v \cup S_v}} \frac{\bw(\sigma_v)^{\val(v)-1+n_v}}{ (\frac{\su_1}{d} - \psi_1)^{n_v} \cdot \prod_{e \in E_v}(\frac{\bw(\tau_e, \sigma_v)}{d_e} - \psi_{(e,v)})} \bigg|_{\su_2 - f\su_1 = 0},
$$
to $\hC_{\vGa} \big|_{\su_2 -f\su_1=0}$. It amounts to showing that
\begin{equation}\label{eqn:OpenRelContributionCompare}
    (-1)^{fd}\frac{\prod_{k = 1}^{d-1}(fd+k)}{ d!\su_1} = (-1)^{d +1}(\su_2-f\su_1) \frac{\bh(\tau_0, d)}{d} \bigg|_{\su_2 - f\su_1 = 0}.
\end{equation}
We compute that if $f\ge 0$,
\begin{equation}\label{eqn:bhTauZeroPositive}
    \begin{aligned}
        (\su_2-f\su_1)\bh(\tau_0, d)\bigg|_{\su_2 - f\su_1 = 0} 
        =& \frac{d^d}{d!\su_1^d} \cdot 
            \frac{(-\frac{d-1}{d}\su_1 - \su_2)(-\frac{d-2}{d}\su_1 - \su_2) \cdots (\frac{df-1}{d}\su_1 - \su_2)}{\su_2(-\frac{\su_1}{d}+\su_2)(-\frac{2\su_1}{d}+\su_2) \cdots (-\frac{df-1}{d}\su_1+\su_2)} \bigg|_{\su_2 - f\su_1 = 0}\\
        =& \frac{d^d}{d!\su_1} \cdot
            \frac{(-\frac{d-1}{d} - f)(-\frac{d-2}{d} - f) \cdots (\frac{df-1}{d} - f)}{f(-\frac{1}{d} + f)(-\frac{2}{d}+f) \cdots (-\frac{df-1}{d}+f)}\\
        =& \frac{(-1)^{(f+1)d+1} d}{d!\su_1} \cdot \prod_{k=1}^{d-1}(fd + k);
    \end{aligned}
\end{equation}
on the other hand, if $f <0$,
\begin{equation}\label{eqn:bhTauZeroNegative}
    \begin{aligned}
        (\su_2-f\su_1)\bh(\tau_0, d)\bigg|_{\su_2 - f\su_1 = 0}
        =& -\frac{d^d}{d!\su_1^d} \cdot
            \frac{(\frac{\su_1}{d}+\su_2)(\frac{2\su_1}{d}+\su_2) \cdots (-\frac{df+1}{d}\su_1+\su_2)}{(-\su_1 - \su_2)(-\frac{d+1}{d}\su_1 - \su_2)(-\frac{d+2}{d}\su_1 - \su_2) \cdots (\frac{df+1}{d}\su_1 - \su_2)} \bigg|_{\su_2 - f\su_1 = 0}\\
        =& -\frac{d^d}{d!\su_1} \cdot
            \frac{(\frac{1}{d}+f)(\frac{2}{d}+f) \cdots (-\frac{df+1}{d}+f)}{(-1-f)(-\frac{d+1}{d} - f)(-\frac{d+2}{d} - f) \cdots (\frac{df+1}{d} - f)}\\
        =& \frac{(-1)^{(f+1)d+1} d}{d!\su_1} \cdot \prod_{k=1}^{d-1}(fd + k),
    \end{aligned}
\end{equation}
which is the same as the $f \ge 0$ case. Therefore, \eqref{eqn:OpenRelContributionCompare} directly follows.
\end{proof}

\begin{lemma}\label{lem:OpenRelativePositive}
Let $\hbeta \in H_2(Y;\bZ)$ be an effective curve class of $Y$ such that $d := \hbeta \cdot D>0$. Then for any $\vGa \in \Gamma^{Y,D}_{\hbeta} \setminus \Gamma^{Y,D,0}_{\hbeta}$, we have
$$
    \hC_{\vGa} \big|_{\su_2 - f\su_1= 0} = 0.
$$
In particular,
$$
    N^{Y,D}_{\hbeta} = \sum_{\vGa \in \Gamma^{Y,D,0}_{\hbeta}} \hC_{\vGa} \big|_{\su_2-f\su_1 = 0}.
$$
\end{lemma}

We will use the following identity (see e.g. \cite[Lemma 61(a)]{Liu13}):
\begin{equation}\label{eqn:PsiIntegral}
    \int_{\Mbar_{0,n}}\frac{1}{\prod_{j = 1}^n (\bw_j - \psi_j)} = \frac{1}{\bw_1 \cdots \bw_n} \left(\frac{1}{\bw_1} + \cdots + \frac{1}{\bw_n} \right)^{n-3},
\end{equation}
which is consistent with the integration convention \eqref{eqn:UnstableIntegral}.

\begin{proof}
Let $\vGa \in \Gamma^{Y,D}_{\hbeta} \setminus \Gamma^{Y,D,0}_{\hbeta}$. We determine the power of $\su_2-f\su_1$ in $\hC_{\vGa}$. For our generic choice of $f$ (Assumption \ref{assump:GenFraming}), $\bw(\tau_e,\sigma_v) \neq \pm (\su_2-f\su_1)$ for any $(\tau_e,\sigma_v) \in F(\Gamma)$. Then, \eqref{eqn:PsiIntegral} implies that $\su_2-f\su_1$ is not a factor of the denominator of
$$
    \prod_{v \in V(\Gamma)^{(0)}} \int_{\Mbar_{0, E_v}} \frac{1}{\prod_{e \in E_v}(\frac{\bw(\tau_e,\sigma_v)}{d_e} - \psi_{(e,v)})}.
$$
It suffices to focus on the term
$$
    (\su_2-f\su_1)^{2\ell(\mu(\vGa))-1}  \prod_{e \in \vf^{-1}(\tau_0)}\bh(\tau_e, d_e).
$$
By a computation similar to \eqref{eqn:bhTauZeroPositive} and \eqref{eqn:bhTauZeroNegative}, the power of $\su_2-f\su_1$ in $\bh(\tau_0, a)$ is $-1$ for any $a \in \bZ_{>0}$. Therefore, the total power of $\su_2-f\su_1$ is $\ell(\mu(\vGa)) - 1$ 
which is strictly positive since $\vGa \not \in \Gamma^{Y,D,0}_{\hbeta}$. This implies that $\hC_{\vGa} \big|_{\su_2-f\su_1=0} = 0$.
\end{proof}

\subsection{Relative/local correspondence}\label{sect:RelLocal}
In this section, we identify the maximally-tangent relative Gromov-Witten invariants of $(Y,D)$ and the closed Gromov-Witten invariants of $\tX = \cO_Y(-D)$ in genus zero. Recall from Section \ref{sect:Homology} that the natural isomorphism $H_2(Y;\bZ) \cong H_2(\tX;\bZ)$ identifies effective curve classes.

\begin{theorem}[Relative/local correspondence]\label{thm:RelativeLocal}
Let $\hbeta \in H_2(Y; \bZ)$ be an effective class of $Y$ and $\tbeta \in H_2(\tX;\bZ)$ be the corresponding effective class of $\tX$, such that $d:= \hbeta \cdot D = \tbeta \cdot \tD >0$. Then
\begin{equation}\label{eqn:RelativeLocalThm}
     N^{Y,D}_{\hbeta} = (-1)^{d + 1}d N^{\tX}_{\tbeta,0} = (-1)^{d + 1} N^{\tX}_{\tbeta,1}.
\end{equation}
\end{theorem}

By Lemma \ref{lem:DivisorAxiom}, to be proven shortly, we will mainly focus on the correspondence between $N^{Y,D}_{\hbeta}$ and $N^{\tX}_{\tbeta,0}$. Based on Propositions \ref{prop:RelativeLocalResult}, \ref{prop:ClosedLocalResult} and Lemma \ref{lem:OpenRelativePositive}, we first show in Lemma \ref{lem:RelativeLocalVal1} that $N^{Y,D}_{\hbeta}$ is in correspondence with the total contribution to $N^{\tX}_{\tbeta,0}$ from decorated graphs in
$$
    \Gamma^0_{0,0}(\tX, \tbeta):= \{\vGa \in \Gamma_{0,0}(\tX, \tbeta): |\vf^{-1}(\iota(\tau_0))| =1\}.
$$
Each $\vGa = (\Gamma, \vf, \vd) \in \Gamma^0_{0,0}(\tX, \tbeta)$ satisfies the following properties:
\begin{itemize}
    \item There is a unique vertex $\hv_0 = \hv_0(\vGa) \in V(\Gamma)$ such that $\vf(\hv_0) = \iota(\hsi_0)$.
    
    \item There is a unique edge $e_0 = e_0(\vGa) \in E(\Gamma)$ such that $\vf(e_0) = \iota(\tau_0)$. We have $(e_0, \hv_0) \in F(\Gamma)$ and $\vd(e_0) = d$.
\end{itemize}
Then, in Lemma \ref{lem:RelativeLocalValLarge}, we show that there is no contribution from decorated graphs in $\Gamma_{0,0}(\tX, \tbeta) \setminus \Gamma^0_{0,0}(\tX, \tbeta)$ after the restriction $\su_4 = 0$, $\su_2 - f\su_1 = 0$.

For each $\vGa \in \Gamma_{0,0}(\tX, \tbeta)$, we set
$$
    \tC_{\vGa}:= \frac{\su_2-f\su_1}{|\Aut(\vGa)|} \prod_{e \in E(\Gamma)} \frac{\tbh(\ttau_e, d_e)}{d_e} \prod_{v \in V(\Gamma)} \int_{\Mbar_{0, E_v}} \frac{\tbw(\tsi_v)^{\val(v)-1}}{\prod_{e \in E_v}(\frac{\tbw(\ttau_e,\tsi_v)}{d_e} - \psi_{(e,v)})}
$$
to be the contribution of $\vGa$ to $N^{\tX}_{\tbeta,0}$ as in \eqref{eqn:ClosedLocalResult0} in Proposition \ref{prop:ClosedLocalResult} before the restriction $\su_4 = 0, \su_2 - f\su_1 = 0$.

\begin{lemma}\label{lem:RelativeLocalVal1}
For $\hbeta, \tbeta$ as in Theorem \ref{thm:RelativeLocal}, we have
\begin{equation}\label{eqn:RelativeLocalVal1}
    \frac{(-1)^{d + 1}}{d} N^{Y,D}_{\hbeta} = \sum_{\vGa \in \Gamma^0_{0,0}(\tX, \tbeta)} \tC_{\vGa} \big|_{\su_4 = 0, \su_2 - f\su_1 = 0}.
\end{equation}
\end{lemma}

\begin{proof}
We use the computation of $N^{Y,D}_{\hbeta}$ in Proposition \ref{prop:RelativeLocalResult} and the vanishing result Lemma \ref{lem:OpenRelativePositive}. First, note that the bijective map $\iota: \hSi(2)_c \sqcup \hSi(3) \to \tSi(3)_c \sqcup \tSi(4)$ of cones induces a natural bijection
$$
    \Gamma^{Y,D,0}_{\hbeta} \longleftrightarrow \Gamma^0_{0,0}(\tX, \tbeta)
$$
where, given $\vGa \in \Gamma^{Y,D,0}_{\hbeta}$, the corresponding decorated graph in $\Gamma^0_{0,0}(\tX, \tbeta)$ is formed by forgetting the marked point and post-composing the label map $\vf$ with the map $\iota$ of cones.

Now let $\vGa \in \Gamma^{Y,D,0}_{\hbeta}$ and $\vGa' \in \Gamma^0_{0,0}(\tX, \tbeta)$ be a pair of corresponding decorated graphs. Let $\hv_0 = \hv_0(\vGa) = \hv_0(\vGa')$ and $e_0 = e_0(\vGa) = e_0(\vGa')$. Note that $\Aut(\vGa) = \Aut(\vGa')$. We compare the contribution $\tC_{\vGa'}$ to the contribution $\hC_{\vGa}$ defined in \eqref{eqn:RelContribution}.

We first determine the power of $\su_4$ in $\tC_{\vGa'}$. We start by considering the term
$$
    \prod_{e \in E(\Gamma)} \tbh(\ttau_e, d_e) \prod_{v \in V(\Gamma)} \tbw(\tsi_v)^{\val(v)-1}.
$$
From \eqref{eqn:Sigma0Wts}, \eqref{eqn:TanWtsEqual}, \eqref{eqn:Sigma0TanWts}, we see that for any $\hsi \in \hSi(3)$ with $\hsi \neq \hsi_0$, $\su_4$ has power $1$ in $\tbw(\iota(\hsi))$ and
\begin{equation}\label{eqn:tbwSigma}
    \frac{\tbw(\iota(\hsi))}{\su_4} \bigg|_{\su_4 = 0} = \bw(\hsi).
\end{equation}
Moreover, we can compute that for any $a \in \bZ_{>0}$ and $\htau \in \hSi(2)_c$ with $\htau \neq \tau_0$, $\su_4$ has power $-1$ in $\tbh(\iota(\htau), a)$ and
\begin{equation}\label{eqn:tbhTau}
    \su_4 \tbh(\iota(\htau), a) \big|_{\su_4 = 0} = \bh(\htau, a);
\end{equation}
for any $a \in \bZ_{>0}$, $\su_4$ has power $0$ in $\tbh(\iota(\tau_0), a)$ and
\begin{equation}\label{eqn:tbhTauZero}
    \tbh(\iota(\tau_0), a)\big|_{\su_4 = 0} = \bh(\htau, a) \cdot \frac{(-1)^{a} a}{a!\su_1^a} \cdot \prod_{k = 1}^{a-1}(a\su_4 + k \su_1) \bigg|_{\su_4 = 0} = \bh(\htau, a) \cdot \frac{(-1)^{a}}{\su_1}.
\end{equation}
Therefore, the total power of $\su_4$ in $\tC_{\vGa'}$ is
$$
    -|E \setminus \{e_0(\vGa')\}| + \sum_{v \in V(\Gamma)\setminus \{\hv_0\}} (\val(v) -1) = 1 - |E(\Gamma)| + \sum_{v \in V(\Gamma)} (\val(v) -1),
$$
which is zero since the underlying graph $(V(\Gamma), E(\Gamma))$ of $\vGa$ (and $\vGa'$) is a tree.

Note by our integration convention \eqref{eqn:UnstableIntegral} that at $\hv_0$,
$$
    \int_{\Mbar_{0, 1}} \frac{1}{\frac{\tbw(e_0,\hv_0)}{d} - \psi_1} = -\frac{\su_1}{d}.
$$
We then consider
\begin{equation}\label{eqn:VertexIntegralProduct}
    \prod_{v \in V(\Gamma) \setminus \{\hv_0\}} \int_{\Mbar_{0, E_v}} \frac{1}{\prod_{e \in E_v}(\frac{\tbw(\ttau_e,\tsi_v)}{d_e} - \psi_{(e,v)})}.
\end{equation}
For any $(e,v) \in F(\Gamma)$, we have $(\ttau_e,\tsi_v) = (\iota(\htau_e), \iota(\hsi_v))$. Thus by \eqref{eqn:TanWtsEqual}, $\tbw(\ttau_e,\tsi_v) \neq \pm \su_4$. Then, \eqref{eqn:PsiIntegral} implies that $\su_4$ is not a factor of the denominator of \eqref{eqn:VertexIntegralProduct}. Moreover, if $\su_4$ is a factor of the numerator of \eqref{eqn:VertexIntegralProduct}, this means that $\tC_{\vGa'} \big|_{\su_4 = 0} = 0$ and
$$
    \prod_{v \in V(\Gamma) \setminus \{\hv_0\}} \int_{\Mbar_{0, E_v}} \frac{1}{\prod_{e \in E_v}(\frac{\bw(\htau_e,\hsi_v)}{d_e} - \psi_{(e,v)})} = 0,
$$
which implies that $\hC_{\vGa} = 0$. Thus we assume without loss of generality that $\su_4$ is not a factor of the numerator of \eqref{eqn:VertexIntegralProduct}.

Therefore, the total power of $\su_4$ in $\tC_{\vGa'}$ is zero. We then compute that
$$
    \frac{\tC_{\vGa'}}{\hC_{\vGa}} \bigg|_{\su_4 = 0} = \frac{\tbh(\iota(\tau_0),d)}{\bh(\tau_0,d)} \cdot \int_{\Mbar_{0, 1}} \frac{1}{\frac{\tbw(e_0,\hv_0)}{d} - \psi_1} \bigg|_{\su_4 = 0} = \frac{(-1)^{d}}{\su_1} \cdot \frac{-\su_1}{d} = \frac{(-1)^{d+1}}{d}.
$$
\end{proof}

\begin{lemma}\label{lem:RelativeLocalValLarge}
Let $\tbeta$ be as in Theorem \ref{thm:RelativeLocal}. Then for any $\vGa \in \Gamma_{0,0}(\tX, \tbeta) \setminus \Gamma^0_{0,0}(\tX, \tbeta)$, we have
$$
    \tC_{\vGa} \big|_{\su_4 = 0, \su_2 - f\su_1 = 0} = 0.
$$
In particular, 
$$
    N^{\tX}_{\tbeta,0} = \sum_{\vGa \in \Gamma^0_{0,0}(\tX, \tbeta)} \tC_{\vGa} \big|_{\su_4 = 0, \su_2 - f\su_1 = 0}.
$$
\end{lemma}

\begin{proof}
Let $\vGa \in \Gamma_{0,0}(\tX, \tbeta) \setminus \Gamma^0_{0,0}(\tX, \tbeta)$. We first determine the power of $\su_4$ in $\tC_{\vGa}$. As argued in the proof of Lemma \ref{lem:RelativeLocalVal1} above, $\su_4$ is not a factor of the denominator of
$$
    \prod_{v \in V(\Gamma)} \int_{\Mbar_{0, E_v}} \frac{1}{\prod_{e \in E_v}(\frac{\tbw(\ttau_e,\tsi_v)}{d_e} - \psi_{(e,v)})}.
$$
Focusing on the term
$$
    \prod_{e \in E(\Gamma)} \tbh(\ttau_e, d_e) \prod_{v \in V(\Gamma)} \tbw(\tsi_v)^{\val(v)-1},
$$
we find by \eqref{eqn:tbwSigma}, \eqref{eqn:tbhTau}, \eqref{eqn:tbhTauZero} that the total power of $\su_4$ is
\begin{equation}\label{eqn:u4Power}
    -|E(\Gamma) \setminus \vf^{-1}(\iota(\tau_0))| + \sum_{v \in V(\Gamma) \setminus \vf^{-1}(\iota(\hsi_0))} (\val(v) - 1).
\end{equation}
Since the underlying graph $(V(\Gamma), E(\Gamma))$ of $\vGa$ is a tree, we have
$$
    1 - |E(\Gamma)| + \sum_{v \in V(\Gamma)} (\val(v) -1) = 0.
$$
This implies that
$$
    -|E(\Gamma) \setminus \vf^{-1}(\iota(\tau_0))| + \sum_{v \in V(\Gamma) \setminus \vf^{-1}(\iota(\hsi_0))} (\val(v) - 1) =  -1 + |\vf^{-1}(\iota(\tau_0))| - \sum_{v \in \vf^{-1}(\iota(\hsi_0))} (\val(v)-1) = |\vf^{-1}(\iota(\hsi_0))| -1.
$$
Therefore, if $|\vf^{-1}(\iota(\hsi_0))| \ge 2$, then the power of $\su_4$ in $\tC_{\vGa}$ is strictly positive, which implies that $\tC_{\vGa} \big|_{\su_4 = 0} = 0$.

It remains to consider the case where the power of $\su_4$ in $\tC_{\vGa}$ is zero, which forces $|\vf^{-1}(\iota(\hsi_0))| = 1$. Let $\hv_0$ denote the unique vertex in $V(\Gamma)$ with label $\iota(\hsi_0)$. Note that $E_{\hv_0} = \vf^{-1}(\iota(\tau_0))$, which has size at least 2. We now find the power of $\su_2 - f\su_1$ in $\tC_{\vGa} \big|_{\su_4 = 0}$. 

Recall that $f$ is generic (Assumption \ref{assump:GenFraming}). By \eqref{eqn:TanWtsEqual}, $\tbw(\ttau_e,\tsi_v) \neq \pm (\su_2-f\su_1)$ for any $(\ttau_e,\tsi_v) \in F(\Gamma)$. Then, \eqref{eqn:PsiIntegral} implies that $\su_2-f\su_1$ is not a factor of the denominator of
$$
    \prod_{v \in V(\Gamma)} \int_{\Mbar_{0, E_v}} \frac{1}{\prod_{e \in E_v}(\frac{\tbw(\ttau_e,\tsi_v)}{d_e} - \psi_{(e,v)})} \bigg|_{\su_4 = 0}.
$$
We then focus on the term
$$
    (\su_2-f\su_1) \prod_{e \in E(\Gamma)} \tbh(\ttau_e, d_e) \prod_{v \in V(\Gamma)} \tbw(\tsi_v)^{\val(v)-1} \bigg|_{\su_4 = 0}.
$$
By \eqref{eqn:tSigma0Wts}, we have
$$
    \tbw(\hsi_0) \big|_{\su_4 = 0} =  \su_1^2(\su_2 - f\su_1)^2.
$$
Combining this with \eqref{eqn:bhTauZeroPositive}, \eqref{eqn:bhTauZeroNegative}, \eqref{eqn:tbwSigma}, \eqref{eqn:tbhTau}, \eqref{eqn:tbhTauZero}, we find that the total power of $\su_2 - f\su_1$ is
$$
    1 - |\vf^{-1}(\iota(\tau_0))| + 2(\val(v_0) - 1) = |\vf^{-1}(\iota(\tau_0))| - 1 \geq 1.
$$
Therefore, the power of $\su_2-f\su_1$ in $\tC_{\vGa} \big|_{\su_4 = 0}$ is strictly positive, which implies that $\tC_{\vGa} \big|_{\su_4 = 0, \su_2 - f\su_1 = 0} = 0$.
\end{proof}

Finally, we prove Lemma \ref{lem:DivisorAxiom} which relates the $0$-pointed and $1$-pointed invariants:
\begin{proof}[Proof of Lemma \ref{lem:DivisorAxiom}]
We use the localization computation \eqref{eqn:ClosedLocalResult1} of $N^{\tX}_{\tbeta,1}$ in Proposition \ref{prop:ClosedLocalResult}. Let $\vGa \in \Gamma_{0,1}(\tX, \tbeta)$, and let $\hv_0 = \vs(1) \in V(\Gamma)$. Since $i_{\tsi}^*([\tD]) = 0$ unless $\tsi = \iota(\hsi_0)$, $\vGa$ can contribute to $N^{\tX}_{\tbeta,1}$ only if $\vf(\hv_0) = \iota(\hsi_0)$, which we now assume. Then, since $\vf(e) = \iota(\htau_0)$ for any $e \in E_{\hv_0}$, we have $\bw(\ttau_e, \tsi_{\hv_0}) = -\su_1$. This implies that the integral
$$
    \int_{\Mbar_{0, E_{\hv_0} \cup S_{\hv_0}}} \frac{1}{\prod_{e \in E_{\hv_0}}(\frac{\tbw(\ttau_e,\tsi_{\hv_0})}{d_e} - \psi_{(e,\hv_0})}
$$
is a rational function in $\su_1$ and thus does not have any poles in $\su_4$ or $\su_2 - f\su_1$. By the same argument as in the proof of Lemma \ref{lem:RelativeLocalValLarge}, $\vGa$ can contribute to $N^{\tX}_{\tbeta,1}$ only if it belongs to the subset
$$
    \Gamma^0_{0,1}(\tX, \tbeta):= \{\vGa \in \Gamma_{0,1}(\tX, \tbeta): |\vf^{-1}(\iota(\tau_0))| =1\}.
$$

Note that there is a natural bijection between $\Gamma^0_{0,1}(\tX, \tbeta)$ and $\Gamma^0_{0,0}(\tX, \tbeta)$ given by removing/adding the marked point. Let $\vGa' \in \Gamma^0_{0,0}(\tX, \tbeta)$ be obtained from $\vGa$ by removing the marked point. Note that $\Aut(\vGa') = \Aut(\vGa)$. Let $e_0 \in E(\Gamma)$ be the unique edge in $E_{\hv_0}$. We have $\vd(e_0) = d$. It suffices to show that the contribution of $\vGa$ to $N^{\tX}_{\tbeta,1}$ as in \eqref{eqn:ClosedLocalResult1} before the weight restricitons, which is
\begin{align*}
    & \frac{\su_2-f\su_1}{|\Aut(\vGa)|} \prod_{e \in E(\Gamma)} \frac{\tbh(\ttau_e, d_e)}{d_e}
        \cdot \prod_{v \in V(\Gamma)} \int_{\Mbar_{0, E_v \cup S_v}} \frac{\tbw(\tsi_v)^{\val(v)-1}i_{\tsi_v}^*([\tD])^{n_v}}{\prod_{e \in E_v}(\frac{\tbw(\ttau_e,\tsi_v)}{d_e} - \psi_{(e,v)})}\\
    = & \frac{-\su_1(\su_2-f\su_1)}{|\Aut(\vGa)|} \prod_{e \in E(\Gamma)} \frac{\tbh(\ttau_e, d_e)}{d_e}
        \cdot \int_{\Mbar_{0, \{e_0\} \cup \{1\}}}\frac{1}{-\frac{\su_1}{d} - \psi_{(e_0,\hv_0)}} 
        \cdot \prod_{v \in V(\Gamma)\setminus \{\hv_0\}} \int_{\Mbar_{0, E_v}} \frac{\tbw(\tsi_v)^{\val(v)-1}}{\prod_{e \in E_v}(\frac{\tbw(\ttau_e,\tsi_v)}{d_e} - \psi_{(e,v)})},
\end{align*}
is $d$ times the contribution $\tC_{\vGa'}$ of $\vGa'$ to $N^{\tX}_{\tbeta,0}$ before the weight restricitons. This is true since we have
$$
    -\su_1 \cdot \int_{\Mbar_{0, \{e_0\} \cup \{1\}}}\frac{1}{-\frac{\su_1}{d} - \psi_{(e_0,\hv_0)}} = -\su_1, \qquad \int_{\Mbar_{0, \{e_0\}}}\frac{1}{-\frac{\su_1}{d} - \psi_{(e_0,\hv_0)}} = -\frac{\su_1}{d}
$$
by \eqref{eqn:UnstableIntegral}. Here, we abuse notation and use $\hv_0$ (resp. $e_0$) to also refer to $\hv_0(\vGa')$ (resp. $e_0(\vGa')$).
\end{proof}

\begin{remark}\label{rem:Degeneration}\rm{
Our Theorem \ref{thm:RelativeLocal} can be viewed as an instantiation of the log-local principle of van Garrel-Graber-Ruddat \cite{vGGR19} in the non-compact setting, due to the correspondence between relative and log Gromov-Witten invariants \cite{AMW14}. Similar to \cite{vGGR19}, our numerical correspondence in Theorem \ref{thm:RelativeLocal} can be derived from a cycle-level correspondence proven by degeneration; for this, we may use the degeneration formula of Li \cite{Li02} in relative Gromov-Witten theory, or its variant in log Gromov-Witten theory by Kim-Lho-Ruddat \cite{KLR18}, and equivariant intersection theory \cite{EG98}.
}\end{remark}

\subsection{Open/closed correspondence}\label{sect:OpenClosed}
Our open/relative correspondence (Theorem \ref{thm:OpenRelative}) and relative/local correspondence (Theorem \ref{thm:RelativeLocal}) together imply the following identification of the disk invariants of $(X,L,f)$ and the genus-zero closed Gromov-Witten invariants of $\tX$, proving the \emph{open/closed correspondence} conjectured by \cite{LM01, Mayr01}.

\begin{theorem}[Open/closed correspondence]\label{thm:OpenClosed}
Let $\beta' \in H_2(X;\bZ)$ be an effective class of $X$, $d \in \bZ_{>0}$, and $\tbeta \in H_2(\tX;\bZ)$ be the effective class of $\tX$ corresponding to $\beta' + d[B]$ under the isomorphism $H_2(X,L;\bZ) \cong H_2(\tX;\bZ)$. Then
$$
    N^{X, L, f}_{\beta', d} = dN^{\tX}_{\tbeta,0} = N^{\tX}_{\tbeta,1}.
$$
\end{theorem}

\end{document}